\documentclass[11pt,twoside]{amsart}
\usepackage{amssymb,amsmath}
\usepackage{amscd}
\usepackage{amsthm}
\usepackage{latexsym}
\usepackage[noadjust]{cite}
\usepackage{mathrsfs}
\usepackage{indentfirst}
\usepackage{graphicx}
\usepackage{subfigure}
\usepackage{color}
\usepackage{enumerate}
\usepackage{tikz}
\usepackage{pgfplots}
\pgfplotsset{compat=1.18}
\usepackage{pgfplots}
\pgfplotsset{compat=1.18}
\usepgfplotslibrary{fillbetween}

\def\dfrac{\displaystyle\frac}

\newtheorem{theorem}{Theorem}[section]
\newtheorem{lemma}[theorem]{Lemma}
\newtheorem{proposition}[theorem]{Proposition}
\newtheorem{corollary}[theorem]{Corollary}
\theoremstyle{definition}
\newtheorem{definition}[theorem]{Definition}
\newtheorem{remark}[theorem]{Remark}
\newtheorem{example}[theorem]{Example}

\usepackage[colorlinks,
linkcolor=blue,       
anchorcolor=red,  
citecolor=blue,        
]{hyperref}

\def\Re{{\rm Re}}
\def\R{\mathbb{R}^3}
\def\H{H^1(\R,\mathbb{C})}
\def\HR{H^1(\R,\mathbb{R})}
\def\HA{H_A^1(\R,\mathbb{C})}
\def\I{\int_{\R}}
\def\N{\mathbb{N}}
\def\C{\mathbb{C}}

\def\ds{\displaystyle}
\title[]
{Bound states of  Schr\"odinger-Poisson system in electromagnetic fields}

\author[G. Ping]{Pingkang Guan}
\address{School of Mathematics, Jilin University, Changchun 130012, PR CHINA}
\email{pingkangguan24@gmail.com}

\author[B. Li]{Baihong Li}
\address{School  of Mathematics, Jilin University, Changchun 130012, PR CHINA}
\email{bhlimath@zohomail.com}

\author[D. Mugnai]{Dimitri Mugnai}
\address{Department of Ecological and Biological Sciences, University of Tuscia, Largo dell'Università 01100, Viterbo, Italy}
\email{dimitri.mugnai@unitus.it}

\author[Y. Wei]{Yuanhong Wei}
\address{School of Mathematics, Jilin University, Changchun 130012, PR CHINA}
\email{weiyuanhong@jlu.edu.cn}

\begin{document}

	\begin{abstract}
This paper investigates the Schrödinger–Poisson system involving electromagnetic fields. Our approach is mainly based
on the variational method.  By developing the global compactness result, we establish the existence of higher-energy bound state and ground state solutions under suitable assumptions. In addition, under symmetry conditions, the existence and multiplicity of intertwining solutions  are obtained for both finite and infinite symmetry group.
	\end{abstract}
	\maketitle	
	 {\it Keywords: {\rm Schrödinger–Poisson system, electromagnetic fields, bound state, ground state, intertwining solutions  } }
\\	\hspace{1cm}\\
	{Mathematics Subject Classification (2020):} 35A15, 35B06, 35J20, 35Q40
    \tableofcontents
	
\section{Introduction}

In this paper, we consider the following magnetic Schr\"odinger-Poisson system
\begin{equation}\label{problemSP}
	\left\{
	\begin{aligned}
		&	(-i\nabla+A(x))^2u+V(x)u+K(x)\phi u=|u|^{p-2}u,\quad&& \text{in}~\mathbb{R}^3,\\
		&-\Delta\phi=K(x)|u|^2,\quad&&\text{in}~\mathbb{R}^3,
	\end{aligned}
	\right.
\end{equation}
	where $i$ is the imaginary unit, $4<p<6$, 
	$A:\mathbb{R}^3\rightarrow\mathbb{R}^3$ is the magnetic potential, $V, K:\mathbb{R}^3\rightarrow\mathbb{R}$ are nonnegative functions. Here, the magnetic Schr\"odinger operator is defined as
	\begin{equation*}
		(-i\nabla+A(x))^2:=-\Delta-2iA(x)\cdot\nabla-i\mathrm{div}A(x)+|A(x)|^2.
	\end{equation*}
The problem \eqref{problemSP} is closely related to the time-dependent  magnetic Schr\"odinger-Poisson system
 \begin{equation}\label{time-dependent}
 	    \left\{
        \begin{aligned}
            &i\psi_t=(-i\nabla+A(x))^2\psi+W(x)\psi+K(x)\phi\psi-|\psi|^{p-2}\psi &&\text{in }\mathbb{R}\times\mathbb{R}^3,\\
            &-\Delta\phi=K(x)|\psi|^2&&\text{in }\mathbb{R}\times\mathbb{R}^3,
        \end{aligned}
        \right.
 \end{equation}
 where $W$ is a real electric potential and $A$ is a magnetic potential related to the magnetic field $B=\nabla\times A$.
Here, $\psi: \mathbb{R}\times\mathbb{R}^3\to\mathbb{C}$ denotes the time-dependent wave function,  and $\phi:\mathbb{R}\times\mathbb{R}^3\to\mathbb{R}$ represents an internal potential for a nonlocal self-interaction of the wave function $\psi$, and the nonlinear term describes the interaction effect among particles. In this paper we investigate the standing wave solutions to \eqref{time-dependent} of the form $\psi(t,x)=e^{-i\mu t}u(x)$. Under the substitution $W(x)=V(x)+\mu$, the problem \eqref{time-dependent} reduces to the stationary system \eqref{problemSP}.

When $A(x)\equiv0$, \eqref{problemSP}  becomes the following Schr\"odinger-Poisson system 
\begin{equation}\label{SP}
	\left\{
	\begin{aligned}
		&	-\Delta u+V(x)u+K(x)\phi u=|u|^{p-2}u,\quad&& \text{in}~\mathbb{R}^3,\\
		&-\Delta\phi=K(x)u^2,\quad&&\text{in}~\mathbb{R}^3,
	\end{aligned}
	\right.
\end{equation}
which has attracted considerable attention due to its important applications in quantum mechanics and semiconductor physics \cite{A1996, B1981}. D'Aprile-Mugnai \cite{DM20041, DM2004} and later Ruiz \cite{R2006} considered the problem
\begin{equation*}
	\left\{
	\begin{aligned}
		&	-\Delta u+u+\lambda\phi u=u^p,\quad&& \text{in}~\mathbb{R}^3,\\
		&-\Delta\phi=u^2,\quad&& \text{in}~\mathbb{R}^3,
	\end{aligned}
	\right.
\end{equation*}
where $\lambda>0, 1<p<5$.  Depending on the parameters $\lambda$ and $p$, the existence and nonexistence results were obtained by Nehari-Pohozaev identity. When $V(x)\equiv1$, Cerami and Vaira \cite{CV2010} studied the system \eqref{SP} with the nonlinearity $|u|^{p-2}u$ replaced by $a(x)|u|^{p-2}u$, where $K(x)$ and $a(x)$ are nonnegative functions satisfying certain asymptotic behaviors. By means of a compactness lemma, they established both a ground state solution and a higher-energy bound state solution. Mugnai \cite{M2011} invesigated the Schr\"odinger-Poisson system
\begin{equation*}
	\left\{
	\begin{aligned}
		&	-\Delta u+\omega u-\lambda\phi u+W_u(x,u)=0,\quad&& \text{in}~\mathbb{R}^3,\\
		&-\Delta\phi=u^2,\quad&& \text{in}~\mathbb{R}^3,
	\end{aligned}
	\right.
\end{equation*}
where $W$ is a positive potential that generalizes the pure power case $|u|^p/p$ with $p>2$. Under suitable assumptions on $W$ and depending on the parameters $\lambda$, $\omega$, the results on the existence, nonexistence, positivity, and multiplicity of solutions are given.
For further research on the Schr\"odinger-Poisson system, we refer the readers to previous studies \cite{A2008, HZ2012, AP2008, WZ2007,YSD2009, JFZ2017, CM2019} and the references therein.

When $A(x)\not\equiv0$, the magnetic Schr\"odinger equation has also attracted considerable attention in recent years. Esteban and Lions \cite{EL1989} studied the following problem
\begin{align*}
    (-i\nabla+A(x))^2u+V(x)u=f(|u|^2)u.
\end{align*}
For the cases $N=2,3$, they established the existence of solutions by applying the concentration compactness lemma to suitable minimization problems. 
 Arioli and Szulkin \cite{AS2003} considered both the Sobolev critical and subcritical cases.
In the critical case, they employed a constrained minimization method, and for the subcritical case, they used a minimax argument to obtain the existence of nontrivial solutions under the assumption that $\sigma(-\Delta_A+V)\subset(0,\infty)$. Liu et al. \cite{LLJ2021} considered the magnetic Schr\"odinger-Poisson type equation
\begin{equation}\label{hartree}
     \left(\frac{\epsilon}{i}\nabla-A(x)\right)^2u+V(x)u+\epsilon^{-2}(|x|^{-1}*u^2)u=f(|u|^2)u,\quad\text{in}~\mathbb{R}^3.
\end{equation}
By combining variational methods, penalization techniques, and Ljusternik–Schnirelmann theory, they proved the multiplicity and concentration of solutions for sufficiently small $\epsilon>0$.
For more research on nonlinear partial differential equations involving magnetic fields, we refer the readers to \cite{DL2013, AF2011, CS2005, DAprileWei, DW2011, dJ2022, LZ2023, MJ2021, JZR2022, JR2020, Ad2018, A2010} and the references therein.

Motivated by the aforementioned works, the present paper is devoted to investigating the magnetic Schr\"odinger-Poisson system \eqref{problemSP} under the following hypotheses on $A$, $V$ and $K$ respectively:
\begin{enumerate}
	\item[$(A_1)$]$A\in L^\infty(\mathbb{R}^3,\mathbb{R}^3)$, $\lim\limits_{|x|\rightarrow\infty}A(x)=0$.
	\item[$(V_1)$]$V\in L^\infty(\mathbb{R}^3,\mathbb{R})$, $V\geq0$,  $\lim\limits_{|x|\rightarrow\infty}V(x)=V_\infty>0$.
     \item [$(K_1)$] $K\in L^2(\mathbb{R}^3,\mathbb{R})$, $K\geq0$, $\lim\limits_{|x|\rightarrow\infty}K(x)=0$.
\end{enumerate}

It is worth noticing that the presence of the potential $K$ vanishing at infinity is a novelty, if compared to the previous papers. Indeed, in the cited papers the related Hartree equation, like the one in \eqref{hartree}, can be re-written as system \eqref{problemSP} with $K(x)\equiv 1$ - or as similar systems with different powers but always with $K(x)\equiv 1$, like in \cite{LZ2023}. Thus, looking at the second equation in \eqref{problemSP}, with the suitable functional setting to be described below, we have
\[
\int_{\R}|\nabla \phi|^2dx=\int_{\R}K(x)|u|^2\phi\,dx.
\]
If, for instance, $K\equiv0$ in some region $\mathcal F$, we have that the norm of $\phi$ in $\R$ is somehow concentrated on the support of $K$, losing all the related information in $\mathcal F$: if $K\equiv 0$ out of a ball, we lose all the information at infinity, which is unnatural in the classical Sobolev setting. However, with the suitable spaces, a variational approach is still valid.

For this, we set
	$$\nabla_Au:=-i\nabla u+A(x)u$$ 
and introduce the Hilbert spaces $	\HA$ defined by
	\begin{equation*}
	\HA:=\left\{u\in L^2(\R,\mathbb{C}):\nabla_Au\in L^2(\R,\mathbb{C}^3)\right\}
\end{equation*}
and $\mathcal{D}^{1,2}(\mathbb{R}^3,\mathbb{R})$ defined by
\begin{equation*}
    \mathcal{D}^{1,2}(\mathbb{R}^3,\mathbb{R}):=\{u\in L^6(\R,\mathbb{R}):\nabla u\in L^2(\R,\mathbb{R}^3)\},
\end{equation*}
which will be described better in Section \ref{sec2}.

We will prove that there exists a unique $\phi=\phi_{|u|}\in\mathcal{D}^{1,2}(\mathbb{R}^3,\mathbb{R})$ such that
    \begin{align*}
        -\Delta\phi=K(x)|u|^2;
    \end{align*}
thus, $(u,\phi_{|u|})\in \HA\times \mathcal{D}^{1,2}(\mathbb{R}^3,\mathbb{R})$ is a weak solution of \eqref{problemSP} if and only if $u\in \HA$ is a weak solution of the following equation
    \begin{equation*}\label{problem}\tag{$P$}
        (-i\nabla+A(x))^2u+V(x)u+K(x)\phi_{|u|} u=|u|^{p-2}u,\quad\text{in}~\mathbb{R}^3.
    \end{equation*}
    Hereafter, we consider problem \eqref{problem} instead of \eqref{problemSP}.
    
In the sequel, we investigate the existence of {\sl bound state} solutions of \eqref{problem}, namely  solutions of \eqref{problem} with finite energy. In particular, among bound states, a nontrivial solution with the least energy  is called a 
 {\sl ground state}.
 
Now, we give the following qualitative results, in which the assumptions of potentials $A$, $V$ and $K$ are introduced to ensure the existence of a bound state solution.
\begin{theorem}\label{cor:fixed-AV}
Assume that $(A_1)$, $(V_1)$ and $(K_1)$ hold. Then there exists $\Lambda_0>0$ such that for any $\bar\Lambda\in(0,\Lambda_0)$ there exists 
$\delta=\delta(\bar\Lambda)>0$  such that if $A$ and $V$ satisfy
\begin{equation*}
V_\infty<|A(x)|^2+V(x)<V_\infty+\bar\Lambda
\quad \text{for } x\in\mathbb{R}^3,
\end{equation*}
and
\begin{equation*}
    |K|_2<\delta,
\end{equation*}
then $(P)$ has a bound state solution.
\end{theorem}

\begin{theorem}\label{cor:fixed-K}
Assume that $(A_1)$, $(V_1)$ and $(K_1)$ hold. Then there exists $s_0>0$ such that for any $K$ satisfying $|K|_2^2\in(0,s_0)$, there exists $\Lambda_K>0$ such that if $A$ and $V$ satisfy
\begin{equation*}
    V_\infty<|A(x)|^2+V(x)<V_\infty+\Lambda_K
\quad \text{for } x\in\mathbb{R}^3,
\end{equation*}
then \eqref{problem} has a bound state solution.
\end{theorem}

\begin{remark}
The precise value for $s_0$ is  given implicitly in \eqref{s0} below.
\end{remark}

\begin{remark}
The two results above describe the relationship among the potentials, which guarantees the existence of bound state solutions and we point out that in both cases, $A$ and $V$ act on \eqref{problem} in a cooperative way, and they jointly compete with $K$.
\end{remark}

The energy functional $E_A:\HA\rightarrow\mathbb{R}$  corresponding to \eqref{problem} is  defined by
\begin{align*}
	E_A(u):=\dfrac{1}{2}\int_{\mathbb{R}^3}\left(|\nabla_Au|^2+V(x)|u|^2\right)dx+\frac{1}{4}\int_{\mathbb{R}^3}K(x)\phi_{|u|}|u|^2\,dx-\dfrac{1}{p}\int_{\mathbb{R}^3}|u|^pdx.
\end{align*}
As usual, one proves that $E_A\in C^1(\HA,\mathbb{R})$, and that
\begin{align*}
    E_A'(u)v=\Re\left(\int_{\mathbb{R}^3}\left(\nabla_A u\cdot\overline{\nabla_A v}+V(x)u\overline{v}\right)\,dx+\int_{\mathbb{R}^3}K(x)\phi_{|u|}u\overline{v}\,dx-\int_{\mathbb{R}^3}|u|^{p-2}u\overline{v}\,dx\right).
\end{align*}
Moreover, one can easily verify that $E_A$ is unbounded from below in $\HA$, so a minimizing procedure can not be used directly.
Hence, it is natural to restrict $E_A$ to a natural
constraint.

In particular, we define the Nehari manifold
\begin{equation*}
	\mathcal{N}_A:=\{u\in\HA\setminus\{0\}:E_A^\prime(u)u=0\}
\end{equation*}
and the associated least energy
\begin{equation*}
	m_A:=\inf_{u\in\mathcal{N}_A}E_A(u).
\end{equation*}

 Next, we introduce the Sobolev space
\begin{align*}
    H^1(\mathbb{R}^3,\mathbb{C}):=\left\{u\in L^2(\mathbb{R}^3,\mathbb{C}):\nabla u\in L^2(\mathbb{R}^3,\mathbb{C}^3)\right\},
\end{align*}
equipped with 
the norm
\begin{align*}
    \|u\|:=\left(\int_{\mathbb{R}^3}\left(|\nabla u|^2+V(x)|u|^2\right)dx\right)^{\frac{1}{2}}.
\end{align*}
 We also need to consider the following autonomous limit problem 
\begin{equation*}\tag{\textit{$P_\infty$}}
	\left\{
	\begin{aligned}
		&-\Delta u+V_\infty u=|u|^{p-2}u,\quad \mathrm{in}~\mathbb{R}^3,\\
		&u\in H^1(\mathbb{R}^3,\mathbb{C}).
	\end{aligned}
	\right.
	\end{equation*}
    The energy functional $E_0:H^1(\mathbb{R}^3,\mathbb{C})\rightarrow\mathbb{R}$ associated with \eqref{limitproblem} is given by
	\begin{equation*}
			E_0(u):=\dfrac{1}{2}\int_{\mathbb{R}^3}\left(|\nabla u|^2+V_\infty|u|^2\right)dx-\dfrac{1}{p}\int_{\mathbb{R}^3}|u|^pdx.	
	\end{equation*}
Similarly, we define the Nehari manifold
	\begin{equation*}
		\mathcal{N}_0:=\{u\in H^1(\mathbb{R}^3,\mathbb{C})\setminus\{0\}:E_0^\prime(u)u=0\}
	\end{equation*}
	and the corresponding least energy 
	\begin{equation*}
m_0:=\inf_{u\in\mathcal{N}_0}E_0(u).
	\end{equation*}

From \cite{BL1983}, we know that there exists  a real-valued and radially decreasing symmetric ground state solution $\omega$ of the following {\sl real-valued} problem, obviously strictly related to $(P_\infty)$,
\begin{equation*}\tag{\textit{$P_{\infty,\mathbb{R}}$}}
	\left\{
	\begin{aligned}
		&-\Delta u+V_\infty u=|u|^{p-2}u,\quad \mathrm{in}~\mathbb{R}^3,\\
		&u\in H^1(\mathbb{R}^3,\mathbb{R}).
	\end{aligned}
	\right.
\end{equation*}
Actually, $\omega$ is also a ground state solution of \eqref{limitproblem}, see Proposition \ref{thesamelevel}.

The next Lemma is proved in  Bonheure, Nys and Van Schaftingen \cite[Theorem 1]{BNV2019}, and it states that the least energy solutions of \eqref{limitproblem} is unique up to translations and a complex phase. It also indicates that $m_0$ is an isolated critical level of $E_0$.
\begin{lemma}\label{isolatedlevel}
	There exists $\xi>0$ such that if $v$ is a critical point of $E_0$ satisfying $E_0(v)\leq m_0+\xi$, then $v=e^{i\theta}\omega(\cdot-a)$ for some $a\in\mathbb{R}^3$ and $ \theta\in\mathbb{R}$.
\end{lemma}

For future calculations, we define
\begin{equation}\label{mhat}
\hat{m}:=\min\{m_0,\xi\}.
\end{equation}

Let $S$ be the best constant for the Sobolev embedding $\mathcal{D}^{1,2}(\mathbb{R}^3,\mathbb{R})\hookrightarrow L^6(\mathbb{R}^3,\mathbb{R})$
\begin{equation}\label{S}
     S:=\inf_{u\in\mathcal{D}^{1,2}(\mathbb{R}^3,\mathbb{R})\setminus\{0\}}\frac{\Vert u\Vert_{\mathcal{D}^{1,2}}^2}{|u|_6^2}.
 \end{equation}
 Besides, for every $u\in H^1(\mathbb{R}^3, \mathbb{R})$, we define the norm
 \[
\|u\|_\bullet^2:=\int_{\R}(|\nabla u|^2+V_\infty|u|^2)dx.
\]

The following theorem provides a quantitative condition involving the relation
between the admissible range of $|A|^2+V$ and the strength of the Poisson coupling potential $K$.

\begin{theorem}\label{highenergy}
	Let  $(A_1)$, $(V_1)$ and  $(K_1)$ hold. Assume that $A$ and $V$ satisfy
    \begin{equation*}
        V_\infty<|A(x)|^2+V(x)<V_\infty+\Lambda\quad f\!or~x\in\R, 
    \end{equation*}
       and $\Lambda$, $K$ satisfy
    \begin{equation}\label{AE}
        \left(\Lambda+\frac{S^{-3}\left| K\right|_2^2\Vert\omega\Vert_A^4}{|\omega|_2^2}\right)\left(V_\infty ^{-1}\Lambda+\frac{S^{-3}\left| K\right|_2^2\Vert\omega\Vert_A^4}{\Vert\omega\Vert^2_\bullet}+1\right)^\frac{4}{p-4}\leq\frac{2\hat{m}}{|\omega|_2^2},
    \end{equation}
then \eqref{problem} has a bound state solution.
\end{theorem}

\begin{remark}\label{bigremark}
The largest possible value for the constant $\Lambda$ appearing in \eqref{AE} depends explicitly on the strength of the Poisson coupling potential $K$. More precisely, setting
$s:=|K|_2^2$, the algebraic equation which defines $\Lambda$ can be
written as
\begin{equation}\label{ae}
    (\Lambda+A_1s)
\left(V_\infty^{-1}\Lambda+A_2s+1\right)^\gamma=A_3,
\end{equation}
where
\begin{equation*}
    A_1=\frac{S^{-3}\|\omega\|_A^4}{|\omega|_2^2},
\qquad
A_2=\frac{S^{-3}\|\omega\|_A^4}{\|\omega\|_\bullet^2},
\qquad A_3=\frac{2\hat m}{|\omega|_2^2},\qquad
\gamma=\frac{4}{p-4}>0.
\end{equation*}
By the implicit function theorem, $\Lambda$ is a $C^1$ function of $s$
and
\begin{equation*}
    \frac{d\Lambda}{ds}=-\frac{A_1\left(V_\infty^{-1}\Lambda+A_2s+1\right)+\gamma A_2(\Lambda+A_1s)}{\left(V_\infty^{-1}\Lambda+A_2s+1\right)+V_\infty^{-1}\gamma(\Lambda+A_1s)}<0.
\end{equation*}
Therefore, $\Lambda$ is strictly decreasing with respect to $|K|_2^2$.
This reflects that as the Poisson coupling potential $K$ becomes stronger ,
 the admissible range for $|A(x)|^2+V(x)$ becomes narrower.
Moreover, if $|K|_2\to0$, then $\Lambda\to\Lambda_0>0$, where
$\Lambda_0$ is uniquely determined by
\begin{equation}\label{Lambda0}
    \Lambda_0\left(V_\infty^{-1}\Lambda_0+1\right)^\gamma=A_3.
\end{equation}
Hence, in the weak Poisson coupling regime ($|K|_2$ small enough), the admissible range for 
$|A(x)|^2+V(x)$ retains a positive width, namely,
\begin{equation*}
 V_\infty<|A(x)|^2+V(x)<V_\infty+\Lambda_0.
\end{equation*}
On the other hand, by \eqref{AE}, $|K|_2$ cannot be too large , but as $|K|_2$ tends to the supremum  $s_0$, which denotes the unique positive root of the equation
\begin{equation}\label{s0}
    A_1s(A_2s+1)^{\gamma}=A_3,
\end{equation} the admissible range for $|A(x)|^2+V(x)$ shrinks to a narrow range near $V_\infty$. In this sense, the strong Poisson coupling
forces the electromagnetic potential $|A|^2+V$ to stay close to its limit $V_\infty$.
This actually reveals a competition between the strength of the
Poisson coupling potential $K$ and the admissible deviation of
$|A|^2+V$ from $V_\infty$.
\end{remark}

  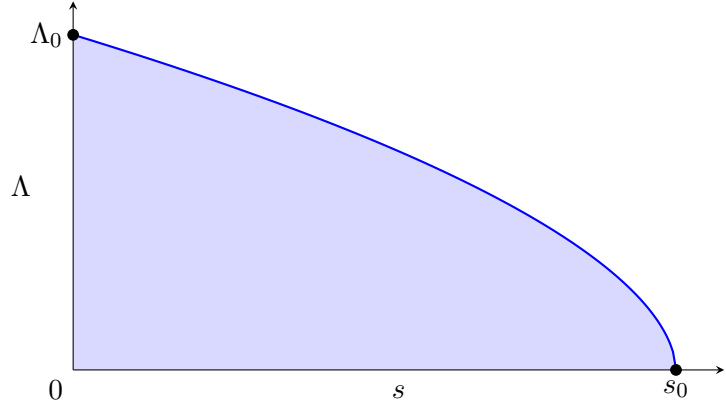
\begin{figure}[ht]
\centering
\begin{tikzpicture}
\begin{axis}[
    width=0.6\textwidth,
    height=0.38\textwidth,
    xlabel={$s$},
    ylabel={$\Lambda$},
    every axis y label/.style={
        at={(axis description cs:-0.08,0.5)},
        anchor=center,
        rotate=0
    },
    xmin=0, xmax=5.4,
    ymin=0, ymax=1.1,
    axis lines=left,
    samples=200,
    domain=0:5,
    xtick=\empty,
    ytick=\empty,
    clip=false
]

\path[name path=xaxis] (axis cs:0,0) -- (axis cs:5,0);

\addplot[thick, blue, name path=curve, domain=0:5]
    {(1-x/5)^0.55};

\addplot[blue!15, draw=none] fill between[
    of=curve and xaxis,
    on layer=axis background
];

\node[left] at (axis cs:0,1) {$\Lambda_0$};
\node[anchor=north] at (axis cs:5,0) {$s_0$};
\node[anchor=north east] at (axis cs:0,0) {$0$};

\addplot[only marks, mark=*] coordinates {(0,1) (5,0)};

\end{axis}
\end{tikzpicture}
\caption{
The curve describes the dependence of $\Lambda$ on $|K|_2^2$ as determined by  \eqref{ae}, and it illustrates the monotonicity of 
$\Lambda$ with respect to the Poisson coupling potential $K$, where $s=|K|_2^2$. 
Besides, from Theorem \ref{cor:fixed-K}, the shadowed region indicates the existence region of a bound state solution for problem \eqref{problem}.}
\label{fig:lambda-K}
\end{figure}

Theorem \ref{highenergy} shows the existence of a bound state solution  when $|A(x)|^2+V(x)$ is confined in the strip region $(V_\infty,V_\infty+\Lambda)$. It is natural to ask whether the bound state solution still exists when $|A(x)|^2+V(x)$ does not lie in this region. In the next result we study this situation, giving a new assumption on the potentials $A, V$ and $K$. This assumption guarantees the existence of a bound state solution, as well, but this bound state is actually a ground state.

\begin{theorem}\label{leastenergy}
	Assume that $(A_1)$, $(V_1)$ and  $(K_1)$ hold.  If $A$, $V$ and $K$ satisfy
	\begin{enumerate}
		\item[$(H_2)$] there exist $c_0>0$, and $0<\sigma<2\sqrt{V_\infty}$ such that $$|A(x)|^2+V(x)\leq V_\infty - c_0e^{-\sigma|x|}\quad f\!or~x\in\R,$$ and $$\int_{\mathbb{R}^3}K(x)e^{2\sqrt{V_\infty}|x|}\,dx<\infty,$$
	\end{enumerate}
	then $\eqref{problem}$ has a ground state solution.
  \end{theorem}
  \begin{remark}\label{integra}
The assumption  on $K$ in $(H_2)$ provides a precise exponential decay on the behavior at infinity. We remark that, as it will be clear from the proof of Theorem \ref{leastenergy}, if  $K\equiv0$, then actually there is no restriction on $\sigma$, but we simply need $\sigma>0$, covering the result of  \cite[Theorem 1.1]{LW2025}. On the other hand, if $K\not \equiv 0$, namely we have a true system, then the upper bound $\sigma<2\sqrt{V_\infty}$ is crucial for our proof.
We also notice that the condition on $K$ in $(H_2)$ implies that $K\in L^1(\R,\mathbb{R})$.
\end{remark}

We would like to point out some difficulties encountered in facing problem \eqref{problem}. Due to the unboundedness of the whole space $\mathbb{R}^3$, it is difficult to recover the compactness of $(PS)$ sequences for $E_A$ in this setting. Inspired by Struwe \cite{S1984}, we establish a global compactness result for problem \eqref{problem}, which plays a crucial role in recovering compactness in suitable energy ranges. First we prove that $m_A\leq m_0,$
see Lemma \ref{le:0<mAleqm0}.
Then we use the global compactness result together with Lemma \ref{isolatedlevel}  to find an energy interval $(m_0,m_0+\hat{m})$,
where $\hat{m}$ is given in \eqref{mhat}, in which  every $(PS)$ sequence for $E_A$ recovers compactness. This allows us to obtain a bound state solution in the case $m_A=m_0$. We emphasize that such a solution is not necessarily a ground state, since its energy may lie above the least energy $m_0$.
Subsequently, we turn to the ground state case. By the global compactness result, $(PS)$ sequences recover compactness in the energy interval $(0,m_0)$. As a consequence, if $m_A<m_0$, then the least energy $m_A$ is achieved and problem \eqref{problem} admits a ground state solution.

Since symmetries naturally arise in various real-world problems, studying solutions of differential equations with specific symmetries is a natural topic. Cingolani and Clapp \cite{CC2009} studied the so-called intertwining solutions of the semi-classical  problem
\begin{equation*}
\left(-\epsilon i\nabla+A(x)\right)^2u+V(x)u=|u|^{p-2}u,\quad x\in\R,
\end{equation*}
 where $N\geq 3$ and $2<p<\frac{2N}{N-2}$.
Here, an {\sl intertwining solution} is a solution which satisfies the following symmetry condition:
\begin{equation}\label{intertwining}
	u(gx)=\eta(g)u(x)\quad \mathrm{for~all}~ g\in G~\mathrm{and} ~x\in\R,
\end{equation}
	where $G$ is a closed subgroup of the group $O(N)$ of orthogonal transformations on $\R$ and $\eta:G\rightarrow \mathbb{S}^1$ is a given group homomorphism into the unit complex numbers. In \cite{CC2009} a multiplicity result for intertwining solutions has been proved for $\epsilon>0$ small enough.	Moreover, Cingolani, Clapp and Secchi \cite{CCS2012} considered the following stationary nonlinear Choquard problem
\begin{equation*}
\left(- i\nabla+A(x)\right)^2u+V(x)u=\left(\frac{1}{|x|^\alpha}\ast|u|^p\right)|u|^{p-2}u,\quad x\in\R,
\end{equation*}
where $\alpha\in(0,N)$, $p\in(2-\frac{\alpha}{N},\frac{2N-\alpha}{N-2})$ and $N\geq3$.
They applied the $\mathbb{S}^1$-symmetric mountain pass theorem to obtain infinitely many intertwining solutions involving an infinite group of isometries $G$ in $\R$. Besides, under suitable exponential decay conditions on $A$ and $V$, they obtained an intertwining solution for any $p\in[2,\frac{2N-\alpha}{N-2})$. We also refer to \cite{CS2018} and references therein for further works on  magnetic problems with group action.
	
To the best of our knowledge, the existence  of intertwining solutions of \eqref{problem} has not been investigated under either finite or infinite group actions.  Hence, the second part of this paper  is devoted to to  seek solutions  of \eqref{problem} which satisfy \eqref{intertwining}, where $G$ is a finite or an infinite group.

	Let $G$ be a subgroup of the group $O(3)$, the group  of orthogonal transformations on $\R$,
		We denote the $G$-orbit of $x$ by $Gx:=\{gx:g\in G\}$, its cardinality by $\# Gx$ and 
 the isotropy group of $x$ by $G_x:=\{g_x\in G:g_xx=x\}$.
Let 	$\eta:G\rightarrow \mathbb{S}^1$ be a group homomorphism into the unit complex numbers
  and denote by	${\rm Ker}~\eta:=\{g\in G:\eta(g)=1\}$ the kernel of $\eta$.
 
It is worth noting that solutions satisfying \eqref{intertwining} may be trivial. 
 In fact, for each $x\in\R$, it is possible to choose $g_x\in G$ such that $g_xx=x$, and so 
\begin{equation*}
	u(x)=u(g_xx)=\eta(g_x)u(x);
\end{equation*}
in this way, if $\eta(g_x)\neq 1$, we get $u(x)=0$. Hence, to avoid this case,  we need to impose the following hypothesis: 
\begin{equation*}\tag{$G_\eta$}\label{eq:Geta}
	{\rm there~exists~}x\in\R~{\rm such~that}~G_x\subset{\rm Ker}~\eta.
\end{equation*}
We also need to impose additional natural symmetric assumptions on $A,V$ and $K$:
\begin{enumerate}	\item[$(A_2)$]$A(gx)=gA(x)$ for every $g\in G$.	\item[$(V_2)$]$V(gx)=V(x)$ for every $g\in G$.
     \item [$(K_2)$] $K(gx)=K(x)$ for every $g\in G$.
\end{enumerate}

Likewise, we can use constrained minimization method on symmetric Nehari manifold to obtain  a $(PS)$ sequence, which is also a minimizing sequence. Therefore, the main issue is to show that the symmetric minimizing energy level lies in a range where the compactness of $(PS)$ sequence holds. Inspired by Hirata \cite{Hirata2008}, the precise compactness range is given by the global compactness result and symmetric property.

First, we consider the case in which $G$ is an infinite group,
that is 
\begin{equation*}\tag{$G_1$}
    \#\{gx:g\in G\}=\infty\quad {\rm for~all~}x\in S^2.
\end{equation*}
In this case, the compactness holds at any energy level and we have the following existence result.
\begin{theorem}\label{existence1}
Let $G$ be a closed subgroup of the group $O(3)$ satisfying $(G_\eta)$ and $(G_1)$.	Assume that $(A_1){\rm -}(A_2)$, $(V_1){\rm -}(V_2)$ and  $(K_1){\rm -}(K_2)$ hold.	
		Then \eqref{problem} possesses a ground state intertwining solution.
	\end{theorem}

\begin{example}
 {\rm Let $G=\mathbb{S}^1\times O(1)\subset O(3)$ be a continuous subgroup  on $\R$ and in Theorem \ref{existence1}
 set \begin{equation*}
\eta(e^{i\theta},h)=e^{i\theta}\quad {\rm for~}(e^{i\theta},h)\in \mathbb{S}^1\times O(1).
       \end{equation*}
 Then the solution $u$  that we find satisfies \eqref{intertwining}, that is,
      \begin{equation*}
		u(e^{i\theta}y,hz)=e^{i\theta}u(y,z) \quad{\rm for~all~}(y,z)\in\mathbb{R}^2\times \mathbb{R}.
	\end{equation*}}
    \end{example}
     From a physical perspective, this result ensures the existence of a standing wave owing a nontrivial phase factor under continuous infinite group actions. This solution icorresponds to the ground state of nonlinear quantum systems with magnetic interactions, and it can also be interpreted as a vortex-type solution carrying angular momentum.
	
Concerning the multiplicity of solutions for \eqref{problem}, we simply state the result, referring to Section \ref{secint} for the precise setting. For this, we recall the following
\begin{definition}\label{dist}
Let   $u$ and $v$ be solutions of \eqref{problem}. They are said to be are {\sl geometrically distinct} if $e^{i\theta}u\neq v$ for all $\theta\in\mathbb{R}/2\pi\mathbb{Z}$.
\end{definition}

The multiplicity of intertwining solutions is stated as follows. 

	\begin{theorem}\label{multiplicity}
	Let $G$ be a closed subgroup of the group $O(3)$ satisfying $(G_\eta)$ and $(G_1)$.
		Assume that $(A_1){\rm -}(A_2)$, $(V_1){\rm -}(V_2)$ and  $(K_1){\rm -}(K_2)$ hold.
		Then \eqref{problem} possesses a sequence of geometrically distinct intertwining solutions $\{u_n\}$	satisfying $E_A(u_n)\rightarrow\infty$ as $n\rightarrow\infty$.
	\end{theorem}
	The above theorem shows that, under infinite group actions, the system admits infinitely many geometrically distinct intertwining solutions with increasing energy levels. Physically,  these solutions represent higher energy bound states carrying nontrivial phase patterns, which are induced by the continuous symmetry and the nonlinear interaction.

We say	 the group $G$ acts effectively on the unit sphere $S^{2}:=\{x\in\R:|x|=1\}$ if
	there exists $g\in G$ such that $gx\neq x$ for $x\in S^{2}$. 
	Next, we consider a finite group $G$ which acts effectively on $S^2$, that is 
    \begin{equation*}\tag{$G_2$}\label{eq:finitegroup}
    \#\{gx:g\in G\}\in [2,\infty)\quad {\rm for~all~}x\in S^2.
\end{equation*}
Define the minimal cardinality of $G$-orbit by
	\begin{equation*}
		\ell(G):=\min_{x\in S^{2}}\left\{\#Gx:g\in G\right\}\in[2,\infty),
	\end{equation*}
so that we can  choose $x_0\in S^{2}$ such that
\begin{equation}\label{x0}
  \ell(G)=\#Gx_0.  
\end{equation}
We also need to introduce $e_j\in S^2$, $j=1,\ldots, \ell(G)$, such that
	\begin{equation}\label{e}
		Gx_0=\left\{e_1,e_2,\cdots,e_{\ell(G)}\right\},
	\end{equation}
so that we can explicitly define
	\begin{equation}\label{dG}
		d_G:=\min_{i\neq j}|e_i-e_j|\in(0,2].
	\end{equation}

For a finite group case, the range of compactness  is weakened, and we show that the compactness holds when the energy level is strictly smaller than $\ell(G)m_0$. In this situation the difficulty is to build a good test--function such that the symmetric least energy lies in the range of compactness. The real-valued function $\omega$, a ground state solution of \eqref{limitproblemR}, plays a significant role in doing this. To be specific, a good test--function is made up by a sum of proper translations of $\omega$. whose components are far away from each
others. Moreover, due to the presence of symmetry, we do not need to use \cite[Theorem 1]{BNV2019} to guarantee that $m_0$ is an isolated critical level of $E_0$.

The existence result of ground state solution is as follows.
	\begin{theorem}\label{existence2}
		Let $G$ be a  finite subgroup of the group $O(3)$ satisfying $(G_\eta)$ and $(G_2)$.
		Assume that $(A_1){\rm -}(A_2)$ $(V_1){\rm -}(V_2)$ and  $(K_1){\rm -}(K_2)$ hold. If $A$, $V$ and $K$ satisfy
        \begin{enumerate}
		\item[$(H_3)$] there exist $c_1>0$ and $\alpha\in(0,d_G\sqrt{V_\infty})$  such that 
	$$|A(x)|^2+V(x)\leq V_\infty - c_1e^{-\alpha|x|}\quad f\!or~x\in\R,$$ and $$\int_{\mathbb{R}^3}K(x)e^{2\sqrt{V_\infty}|x|}\,dx<\infty,$$
    \end{enumerate}
		then \eqref{problem} possesses a ground state intertwining solution.
	\end{theorem}

\begin{example}
Let $G=\{(e^{i\frac{2\pi k}{n}},h): k=0,1,\cdots ,n-1\}$ be a discrete subgroup of $\mathbb{S}^1\times O(1)$ on $\R$.  Set \begin{equation*}
\eta(e^{i\frac{2\pi k}{n}},h)=e^{i\frac{2\pi k}{n}}\quad {\rm for~}(e^{i\theta},h)\in G
       \end{equation*}
       in Theorem  \ref{existence2}. Then we  obtain a solution $u$ which satisfies \eqref{intertwining}, that is,
      \begin{equation*}
		u(e^{i\frac{2\pi k}{n}}y,hz)=e^{i\frac{2\pi k}{n}}u(y,z) \quad{\rm for~all~}k=0,1,\cdots,n-1~{\rm and~}(y,z)\in\mathbb{R}^2\times \mathbb{R}.
	\end{equation*}
\end{example}

We remark that, in contrast with Theorem \ref{existence1}, Theorem \ref{existence2} addresses the case  of finite groups action, which permits the existence of ground state intertwining solutions, as well, with suitable assumptions.

The rest of this paper is organized as follows: In Section 2, we present the variational setting and some preliminaries. The global compactness result is established in Section 3. Section 4 is devoted to the existence of a bound state solution, and Section 5 deals with the existence of a ground state solution. Finally, we show the existence and multiplicity of intertwining solutions in Section 6.

In this paper, we use the following notations:
\begin{itemize}
	\item  $|\cdot|_p$ denotes the norm in  $L^p(\mathbb{R}^3,\mathbb{C})$, $1\leq p\leq\infty$ and $|\cdot|_{L^p(\Omega)}$ is the $L^p$ norm on a measurable set $\Omega$;
	\item $\to$ and $\rightharpoonup$ stand for the strong and weak convergence in the related space;
	\item $C, C_1, C_2, \ldots$ denote positive universal constants which may be different from line to line;
	\item $B_R(y)$ denotes the ball in $\mathbb{R}^3$ with radius $R$ centered at $y$;
    \item $(X^*,\left\|\cdot\right\|_{X^*})$ stands for the dual space of the Banach space  $(X,\left\|\cdot\right\|_X)$.
\end{itemize}

\section{Variational setting and preliminaries}\label{sec2}
In this section, we present the variational framework and some preliminary lemmas.

We start recalling that
	\begin{equation*}
	\HA:=\left\{u\in L^2(\mathbb{R}^3,\mathbb{C}):\nabla_Au\in L^2(\mathbb{R}^3,\mathbb{C}^3)\right\}
	\end{equation*}
is a Hilbert space if equipped with the inner product
	\begin{equation*}
		\langle u,v\rangle_A:=\Re\left(\int_{\mathbb{R}^3}(\nabla_Au\cdot\overline{\nabla_Av}+V(x)u\overline{v})dx\right),
	\end{equation*}
which induces the associated norm
	\begin{equation*}
		\|u\|_A:=\left(\int_{\mathbb{R}^3}\left(|\nabla_Au|^2+V(x)|u|^2\right)dx\right)^{\frac{1}{2}}.
	\end{equation*}
	From \cite[Sections 7.19-7.22]{LL2001},  thanks to $(A_1)$, it follows that $C_c^\infty(\mathbb{R}^3,\mathbb{C})$ is dense in $H_A^1(\mathbb{R}^3,\mathbb{C)}$ and that for each $u\in H_A^1(\mathbb{R}^3,\mathbb{C)}$,  the following diamagnetic inequality hold:
	\begin{equation}\label{diamagnetic}
		|\nabla|u||\leq|\nabla_Au|\quad\mathrm{a.e.~in}~\mathbb{R}^3.
	\end{equation}
	Therefore, we deduce that if $u\in H_A^1(\mathbb{R}^3,\mathbb{C)}$, then $|u|\in H^1(\mathbb{R}^3,\mathbb{R)}$. Combining this  with the Sobolev embedding theorem, we conclude that the embedding
	$H_A^1(\mathbb{R}^3,\mathbb{C)}\hookrightarrow L^t(\mathbb{R}^3,\mathbb{C})$ is continuous for every $t\in[2,6]$, while the embedding $H_A^1(\mathbb{R}^3,\mathbb{C)}\hookrightarrow L^t_{\rm loc}(\mathbb{R}^3,\mathbb{C})$ is compact for every $t\in [1,6)$ by the Rellich-Kondrachov theorem.

We also recall that $\mathcal{D}^{1,2}(\mathbb{R}^3,\mathbb{R})$ is Hilbert space with norm
\[
\|v\|_{\mathcal{D}^{1,2}}:=\left(\int_{\R}|\nabla v|^2dx\right)^{\frac{1}{2}}
\]
and that it is continuously embedded in $L^6(\R,\mathbb{R})$, since by \eqref{S}
\begin{equation}\label{D6}
|\phi|_6\leq S^{-\frac{1}{2}}|\nabla \phi|_2 \quad \mbox{ for every }\phi\in \mathcal{D}^{1,2}(\mathbb{R}^3,\mathbb{R}).
\end{equation}

Since for any $u\in H_A^1(\mathbb{R}^3,\mathbb{C)}$, we have $|u|\in H^1(\mathbb{R}^3,\mathbb{R})$, we can consider the linear functional $\mathcal{L}_{|u|}:\mathcal{D}^{1,2}(\mathbb{R}^3,\mathbb{R})\to\mathbb{R}$ defined by
    \begin{align*}
        \mathcal{L}_{|u|}(v):=\int_{\mathbb{R}^3}K(x)|u|^2v\,dx,\quad\forall v\in \mathcal{D}^{1,2}(\mathbb{R}^3,\mathbb{R}).
    \end{align*}
From the H\"older inequality, the diamagnetic inequality and \eqref{D6}, one has
    \begin{equation}\label{Lu}
        \left| \mathcal{L}_{|u|}(v)\right|\leq\left|K\right|_2| u|_6^2| v|_6\leq S^{-\frac{3}{2}}|K|_2\Vert |u|\Vert^2_{\mathcal{D}^{1,2}}\Vert v\Vert_{\mathcal{D}^{1,2}}\leq S^{-\frac{3}{2}}\left|K\right|_2\left\|u\right\|_A^2\|v\|_{\mathcal{D}^{1,2}}.
    \end{equation}
 Hence $ \mathcal{L}_{|u|}$ is a continuous linear operator on $\mathcal{D}^{1,2}(\mathbb{R}^3,\mathbb{R})$; thus, by the Lax-Milgram theorem, there exists a unique $\phi=\phi_{|u|}\in\mathcal{D}^{1,2}(\mathbb{R}^3,\mathbb{R})$ which weakly solves
    \begin{align*}
        -\Delta\phi=K(x)|u|^2,
    \end{align*}
namely
\begin{equation}\label{debole}
\int_{\R}\nabla \phi\cdot \nabla v\,dx=\int_{\R} K(x)|u|^2v\,dx \quad \mbox{ for all }v\in \mathcal{D}^{1,2}(\mathbb{R}^3,\mathbb{R}).
\end{equation}
Moreover, $\phi_{|u|}$ has the following representation formula
\[
\phi_{|u|}(x)=\frac{1}{4\pi}\int_{\mathbb{R}^3}\frac{K(y)|u(y)|^2}{|x-y|}\,dy.
\]
Besides, by \eqref{Lu} and \eqref{debole}, we have
\begin{equation}\label{normphi}
\|\phi_{|u|}\|_{\mathcal{D}^{1,2}}^2 = \int_{\mathbb{R}^3}K(x)\phi_{|u|}|u|^2\,dx\leq S^{-\frac{3}{2}}|K|_2\Vert |u|\Vert^2_{\mathcal{D}^{1,2}}\Vert \phi_{|u|}\Vert_{\mathcal{D}^{1,2}}\leq S^{-\frac{3}{2}}\left|K\right|_2\Vert u\Vert_A^2\Vert\phi_{|u|}\Vert_{\mathcal{D}^{1,2}},
\end{equation}
and so
\begin{equation}\label{S-1}
\|\phi_{|u|}\|_{\mathcal{D}^{1,2}}\leq  S^{-\frac{3}{2}}|K|_2\Vert |u|\Vert^2_{\mathcal{D}^{1,2}} \leq S^{-\frac{3}{2}}\left|K\right|_2\Vert u\Vert_A^2.
\end{equation}

The next lemma provides some properties of $\phi_{|u|}$; the proof is similar to that of \cite[Proposition 2.2]{CM2016}, but for the sake of completeness, we include the details.
\begin{lemma}\label{propertyofphi}
    Assume that $u_n\rightharpoonup u$ in $\HA$. Then
    \begin{enumerate}[{\rm (a)}]
        \item $\phi_{|u_n|}\rightarrow\phi_{|u|}$ in $\mathcal{D}^{1,2}(\mathbb{R}^3,\mathbb{R})$;
                \item $\ds\I K(x)\phi_{|u_n|}|u_n|^2dx\rightarrow\I K(x)\phi_{|u|}|u|^2dx$;
            \item $\ds\Re\left(\I K(x)\phi_{|u_n|}u_n\overline{\varphi}dx\right)\rightarrow\Re\left(\I K(x)\phi_{|u|}u\overline{\varphi}dx\right)$ for any $\varphi\in\HA$.
    \end{enumerate}
\end{lemma}
\begin{proof}
\hspace{1cm}
    \begin{enumerate}[{\rm (a)}]
        \item For each $u\in\HA$, the Lax-Milgram theorem implies that
        \begin{equation*}
    \|\phi_{|u|}\|_{\mathcal{D}^{1,2}}=\|\mathcal{L}_{|u|}\|_{(\mathcal{D}^{1,2})^*},
        \end{equation*}
        Thus, it is sufficient to show that as $n\to\infty$,
        \begin{equation}\label{eq:mathcalLun}
        \|\mathcal{L}_{|u_n|}-\mathcal{L}_{|u|}\|_{(\mathcal{D}^{1,2})^*}\to0.
        \end{equation}
        Let $\epsilon>0$. Then there exists $R_\epsilon>0$ large enough such that $ |K|_{L^2(\mathbb{R}^3 \setminus B_{R_\epsilon}(0))}<\epsilon$. For any $v\in\mathcal{D}^{1,2}(\R,\mathbb{R})$, one has by \eqref{D6}
        \begin{equation*}
\begin{aligned}
&\quad\left|\mathcal{L}_{|u_n|}(v) - \mathcal{L}_{|u|}(v)\right|\\
&= \int_{\mathbb{R}^3} K(x)(|u_n|^2 - |u|^2)vdx \\
&\leq \int_{\mathbb{R}^3 \setminus B_{R_\epsilon}(0)} K(x) \left||u_n|^2 - |u|^2\right| |v|dx + \int_{B_{R_\epsilon}(0)} K(x) \left||u_n|^2 - |u|^2\right| |v|  dx \\
&\leq | K|_{L^2(\mathbb{R}^3 \setminus B_{R_\epsilon}(0)} \left||u_n|^2 - |u|^2\right|_3 |v|_6 + \left( \int_{B_{R_\epsilon}(0)} |K(x)|^{\frac{6}{5}} \left||u_n|^2 - |u|^2\right|^{\frac{6}{5}}  dx \right)^{\frac{5}{6}} | v|_6 \\
&\leq \left( C\epsilon  + \left( \int_{B_{R_\epsilon}(0)} |K(x)|^{\frac{6}{5}} \left(|u_n| + |u|\right)^{\frac{6}{5}} \left||u_n| - |u|\right|^{\frac{6}{5}} dx \right)^{\frac{5}{6}} \right) \|v\|_{\mathcal{D}^{1,2}}
\end{aligned}
\end{equation*}
for some $C>0$. Given $M>0$, set $\mathcal{A}_M=\{x\in B_{R_\epsilon}(0):K(x)>M \}$. Since $K\in L^2(\R,\mathbb{R})$, it holds that $|\mathcal{A}_m|\to 0$ as $M\to\infty$, where  $|\mathcal{A}_m|$ denotes the measure of $\mathcal{A}_m$. Thus, for $M$ large enough, we have
\begin{equation*}
    \left( \int_{\mathcal{A}_M} K^2 dx \right)^{\frac{3}{5}}<\epsilon.
\end{equation*}
Moreover, up to a subsequence, we know that 
\begin{equation*}
    u_n\to u\quad {\rm in~}L^{\frac{12}{5}}_{{\rm loc}}(\R,\mathbb{C}).
\end{equation*}
Then
\begin{equation*}
\begin{aligned}
&\int_{B_{R_\epsilon}(0)}  |K(x)|^{\frac{6}{5}} \big(|u_n| + |u|\big)^{\frac{6}{5}} \big||u_n| - |u|\big| ^{\frac{6}{5}}dx \\
&= \int_{\mathcal{A}_M}  |K(x)|^{\frac{6}{5}} \big(|u_n| + |u|\big)^{\frac{6}{5}} \big||u_n| - |u|\big|^{\frac{6}{5}} dx + \int_{B_{R_\epsilon}(0) \setminus \mathcal{A}_M} |K(x)|^{\frac{6}{5}} \big(|u_n| + |u|\big)^{\frac{6}{5}} \big||u_n| - |u|\big|^{\frac{6}{5}} dx \\
&\leq \left( \int_{\mathcal{A}_M} K^2 dx \right)^{\frac{3}{5}} \left( \int_{\mathbb{R}^3}  \big(|u_n| + |u|\big)^6 dx \right)^{\frac{1}{5}} \left( \int_{\mathbb{R}^3}  \big||u_n| - |u|\big|^6 dx \right)^{\frac{1}{5}} \\
&\quad + M^\frac{6}{5} \left( \int_{B_{R_\epsilon}(0)}  \big(|u_n| + |u|\big)^{\frac{12}{5}}  dx \right)^{\frac{1}{2}} \left( \int_{B_{R_\epsilon}(0)}  \big||u_n| - |u|\big|^{\frac{12}{5}} dx \right)^{\frac{1}{2}} \\
&\leq C\epsilon + o(1).
\end{aligned}
\end{equation*}
Therefore,  \eqref{eq:mathcalLun} follows recalling that $\{u_n\}$ is bounded in $\HA$, and so in $L^6(\R,\C)$.
        \item Note that by \eqref{normphi} we have 
\[
\|\phi_{|u_n|}\|_{\mathcal{D}^{1,2}}^2-\|\phi_{|u|}\|_{\mathcal{D}^{1,2}}^2=\I K(x)|u_n|^2\phi_{|u_n|}dx-\I K(x)|u|^2\phi_{|u|}dx.
\]
By point (a) the conclusion follows immediately.
\item It holds that
          \begin{equation*}
        \begin{aligned}
              &\quad\Re\left(\I \left(K(x)\phi_{|u_n|}u_n\overline{\varphi}-K(x)\phi_{|u|}u\overline{\varphi}\right)dx\right)\\
              &= \Re\left(\I K(x)\phi_{|u_n|}\overline{\varphi}\left(u_n-u\right)dx\right)+\Re\left(\I K(x)u\overline{\varphi}\left(\phi_{|u_n|}-\phi_{|u|}\right)\,dx\right).
        \end{aligned}
        \end{equation*}
        From (a) and the H\"older inequality, we have that as $n\to\infty$,
        \begin{equation*}
            \Re \left(\I K(x)\left(\phi_{|u_n|}-\phi_{|u|}\right)u\overline{\varphi}\,dx\right)\to0,
        \end{equation*}
since
\begin{equation}\label{K(x)}
    \left|\I K(x)\left(\phi_{|u_n|}-\phi_{|u|}\right)u\overline{\varphi}\,dx\right|\leq |K|_2|\phi_{|u_n|}-\phi_{|u|}|_6|u|_6|\varphi|_6.
\end{equation}
On the other hand, since $u_n\rightharpoonup u$ in $H_A^1(\R,\C)$, one has that $u_n\rightharpoonup u$ in $ L^6(\R,\mathbb{C})$. Then, recalling that $\phi_{|u_n|}\to \phi_{|u|}$ in $L^6(\R,\mathbb{R})$ by (a) and \eqref{D6}, so that $K\phi_{|u_n|}\bar \varphi\to K\phi_{|u|}\bar \varphi$ in $L^{6/5}(\R,\C)$, we find that  as $n\to\infty$
   \begin{equation*}
     \Re\left(\I K(x)\phi_{|u_n|}\overline{\varphi}\left(u_n-u\right)dx\right)\to 0.  
   \end{equation*}
This completes the proof of (c). 
    \end{enumerate}
\end{proof}
    
Define the $C^1$ energy functional $E_A:H_A^1(\mathbb{R}^3,\mathbb{C})\rightarrow\mathbb{R}$ as
		\begin{align*}
		E_A(u):&=\dfrac{1}{2}\int_{\mathbb{R}^3}\left(|\nabla_Au|^2+V(x)|u|^2\right)dx+\frac{1}{4}\int_{\mathbb{R}^3}K(x)\phi_{|u|}|u|^2\,dx-\dfrac{1}{p}\int_{\mathbb{R}^3}|u|^pdx\\		
		&=\dfrac{1}{2}\|u\|_A^2+\frac{1}{4}\int_{\mathbb{R}^3}K(x)\phi_{|u|}|u|^2\,dx-\dfrac{1}{p}\int_{\mathbb{R}^3}|u|^pdx.
	\end{align*}
	Clearly, weak solutions of \eqref{problem} coincide with critical points of $E_A$. Moreover, all  non-trivial critical points of $E_A$ lie on the Nehari manifold
	\begin{equation*}
		\mathcal{N}_A:=\{u\in H_A^1(\mathbb{R}^3,\mathbb{C)}\setminus\{0\}:E_A^\prime(u)u=0\}.
	\end{equation*}
Let us remark that $E_A$ is bounded from below on $\mathcal{N}_A$. Indeed,
\begin{equation}\label{nehari}
E_A^\prime(u)u=\|u\|_A^2+\int_{\R}K(x)\phi_{|u|}|u|^2dx-\int_{\R}|u|^pdx,
\end{equation}
so that, when $u\in \mathcal{N}_A$, we have
\[
E_A(u)=\left(\frac{1}{2}-\frac{1}{p}\right)\|u\|_A^2+\left(\frac{1}{4}-\frac{1}{p}\right)\int_{\R}K(x)\phi_{|u|}|u|^2dx.
\]
Recalling that $p>4$, we find that $E_A$ is bounded from below on $\mathcal{N}_A$. Therefore, we can define the least energy
		\begin{equation*}
		m_A:=\inf_{u\in\mathcal{N}_A}E_A(u).
	\end{equation*}	
	
Next, we consider the following autonomous limit problem related to \eqref{problem}
	\begin{equation*}\label{limitproblem}\tag{\textit{$P_\infty$}}
	\left\{
	\begin{aligned}
		&-\Delta u+V_\infty u=|u|^{p-2}u,\quad \mathrm{in}~\mathbb{R}^3,\\
		&u\in H^1(\mathbb{R}^3,\mathbb{C}).
	\end{aligned}
	\right.
	\end{equation*}
	The Euler-Lagrange functional $E_0:H^1(\mathbb{R}^3,\mathbb{C})\rightarrow\mathbb{R}$ associated with \eqref{limitproblem} is given by
	\begin{equation*}
			E_0(u):=\dfrac{1}{2}\int_{\mathbb{R}^3}\left(|\nabla u|^2+V_\infty|u|^2\right)dx-\dfrac{1}{p}\int_{\mathbb{R}^3}|u|^pdx.	
	\end{equation*}
	Likewise, we define the Nehari manifold
	\begin{equation*}
		\mathcal{N}_0:=\{u\in H^1(\mathbb{R}^3,\mathbb{C})\setminus\{0\}:E_0^\prime(u)u=0\}
	\end{equation*}
	and the corresponding least energy 
	\begin{equation*}
		m_0:=\inf_{u\in\mathcal{N}_0}E_0(u).
	\end{equation*}
	According to \cite[Lemma 2.6]{DL2013}, we know that $\mathcal{N}_0\neq\emptyset$ and $m_0>0$.
	It is natural to ask whether \eqref{limitproblem} is connected with the following problem
	\begin{equation*}\label{limitproblemR}\tag{\textit{$P_{\infty,\mathbb{R}}$}}
	\left\{
	\begin{aligned}
		&-\Delta u+V_\infty u=|u|^{p-2}u,\quad \mathrm{in}~\mathbb{R}^3,\\
		&u\in H^1(\mathbb{R}^3,\mathbb{R}).
	\end{aligned}
	\right.
\end{equation*}
From \cite{BL1983}, it is well-known that \eqref{limitproblemR} has a positive least energy solution $\omega\in H^1(\mathbb{R}^3,\mathbb{R})$. Furthermore, $\omega$ is unique (up to a translations), radially symmetric decreasing and there exist $C_1,C_2>0$ such that
\begin{equation}\label{decayestimate}
C_1(1+|x|)^{-1}e^{-\sqrt{V_\infty}|x|}\leq|D^i\omega(x)|\leq C_2(1+|x|)^{-1}e^{-\sqrt{V_\infty}|x|},\quad {\rm for~}i=0,1.
\end{equation}

The following result connects the least energy levels of \eqref{limitproblem} and \eqref{limitproblemR}, see \cite[Lemma 2.5]{DL2013} or \cite[Lemma 2.3]{CJS2009}.
\begin{proposition}\label{thesamelevel}
\eqref{limitproblem} and \eqref{limitproblemR} have the same least energy level.
\end{proposition}

By virtue of this result, we immediately know that $\omega$ is also a least energy solution of \eqref{limitproblem}, i.e.  $E_0^\prime(\omega)=0$ and $E_0(\omega)=m_0$. 

The following properties of $\mathcal{N}_A$ are basic. For the purpose of completeness, we show them briefly.

\begin{lemma}\label{properties}
	The following properties of $\mathcal{N}_A$ hold:
	\begin{enumerate}[{\rm (a)}]
		\item There exists a positive number $\rho$ such that $\Vert u\Vert_A>\rho$  for every $u\in\mathcal{N}_A$;
		\item $\mathcal{N}_A$ is a closed $C^1$--submanifold of $H_A^1(\mathbb{R}^3,\mathbb{C})$;
		\item $\mathcal{N}_A$ is a natural constraint of $E_A$;
		\item For each $u\in H_A^1(\mathbb{R}^3,\mathbb{C})$, there exists a unique positive number $t_u$ such that $t_uu\in\mathcal{N}_A$. Moreover, $E_A(t_uu)>E_A(tu)$ for all $t>0$ and $t\neq t_u$. 
	\end{enumerate}
\end{lemma}
		\begin{proof}
		\hspace*{\fill}
\begin{itemize}
\item[(a)] According to the embedding
			$\HA\hookrightarrow L^t(\mathbb{R}^3,\mathbb{C})$ for every $t\in[2,6]$, by \eqref{nehari} we obtain
			\begin{align*}
			0=E_A'(u)u=\Vert u\Vert_A^2+\int_{\mathbb{R}^3}K(x)\phi_{|u|}|u|^2\,dx-\int_{\mathbb{R}^3}|u|^p\,dx\geq \Vert u\Vert_A^2-C\Vert u\Vert_A^p.
			\end{align*}
Since $4<p<6$ and $u\neq0$, we find $\Vert u\Vert_A>\rho $ for some positive number $\rho$.
\item[(b)]  Set $G_A(u):=E_A^\prime(u)u$.
			Noting that $G_A$ is continuous and $\mathcal{N}_A\cup\{0\}=G_A^{-1}(0)$, we deduce  that $\mathcal{N}_A\cup\{0\}$ is closed. From (a) we know that $0$ is an isolated point, so  $\mathcal{N}_A$ is closed in $\HA$, as well. For each $u\in\mathcal{N}_A$, one has
			\begin{align*}
				G_A^\prime(u)u&=2\Vert u\Vert_A^2+4\int_{\mathbb{R}^3}K(x)\phi_{|u|}|u|^2\,dx-p\I|u|^p\,dx\\
                &=-2\Vert u\Vert_A^2+(4-p)\I|u|^p\,dx\\
				&\leq-2\Vert u\Vert_A^2\leq-2\rho^2<0,
			\end{align*}
			which implies that $0$ is a regular value of $G_A(u)$. Since $G_A(u)$ is of class $C^1$, $\mathcal{N}_A$ is a $C^1$-submanifold of $\HA$.
\item[(c)]  Let $u$ be a critical point of $E_A$ on $\mathcal{N}_A$. Due to the Lagrange multiplier theorem, there exists $\lambda\in\mathbb{R}$ such that
			\begin{equation*}
				E_A^\prime(u)-\lambda G_A^\prime(u)=0.
			\end{equation*}
			Then
			\begin{equation*}
				{E_A^\prime(u)u-\lambda G_A^\prime(u)u=0.}
			\end{equation*}
			Since {$E_A^\prime(u)u=0 $ and $ G_A^\prime(u)u<0$} from (b), we get that $\lambda=0$. Hence, it follows that $u$ is a critical point of $E_A$ in $\HA$.
\item[(d)]   For each $u\in \HA\backslash\{0\}$, we consider the fiber mapping
			\begin{equation*}
				\psi(t):=E_A(tu)=\dfrac{t^2}{2}\Vert u\Vert_A^2+\frac{t^4}{4}\int_{\mathbb{R}^3}K(x)\phi_{|u|}|u|^2\,dx-\dfrac{t^p}{p}\I|u|^p\,dx,\quad t>0\footnote{Here we have used that $\phi_{t|u|}=t^2\phi_{|u|}$, which can be proved as in \cite[Proposition 3.1]{DM2004}.}.
			\end{equation*}
			Then
\[
\begin{aligned}
				\psi^\prime(t)
				&= t\Vert u\Vert_A^2+t^3\int_{\mathbb{R}^3}K(x)\phi_{|u|}|u|^2\,dx-t^{p-1}\I|u|^p\,dx \\
				&=\dfrac{E_A^\prime(tu)(tu)}{t}.
\end{aligned}
\]
Hence, $$tu\in\mathcal{N}_A\iff \psi^\prime(t)=0.$$
		Since $4<p<6$, it is easy to verify that for every $u\in \HA$, $\psi^\prime(t)$  has a unique positive root, so that there exists a unique number $t_u$  such that $t_uu\in\mathcal{N}_A$. Moreover, one can easily find out that $t_u$ is a strict maximum point of $\psi$. This completes the proof.  \qedhere
		\end{itemize}
	\end{proof}

\begin{remark}\label{Re.N0property}
It is straightforward to verify that Lemma analogous to Lemma \ref{properties} also holds for $\mathcal{N}_0$.
	\end{remark}

The next lemma applies to dealing with the energy estimates, see \cite[Lemma 2.9]{CM2016}.
\begin{lemma}\label{IntergalLemma}
Let $g\in L^\infty(\mathbb{R}^3)$ and $h\in L^1(\mathbb{R}^3, \mathbb{R})$ be such that, for some $\kappa\geq0$, $\nu\geq0$ and $\gamma\in\mathbb{R}$,
   \begin{align*}
   	\lim_{|x|\to\infty}g(x)e^{\kappa|x|}|x|^\nu=\gamma,\quad\int_{\mathbb{R}^3}|h(x)|e^{\kappa|x|}|x|^\nu dx<\infty.
   \end{align*}
   Then for every $z\in{\mathbb{R}^3\setminus\{0\}}$ we have
   \begin{align*}
   	\lim_{\rho\to\infty}\left(\int_{\mathbb{R}^3}g(x+\rho z)h(x)\,dx\right)e^{\kappa|\rho z|}|\rho z|^\nu =\gamma\int_{\mathbb{R}^3}h(x)e^{-\frac{\kappa(x\cdot z)}{|z|}}\,dx.
   \end{align*}
\end{lemma}
	
	\section{A global compactness result}
	Recall that a sequence $\{u_n\}\subset \HA$ is said to be a $(PS)_d$ sequence of $E_A$ at some level $d\in\mathbb{R}$, if $E_A(u_n)\rightarrow d$ and $E_A^\prime(u_n)\rightarrow 0$. Moreover, the functional $E_A$ is said to satisfy the $(PS)_d$ condition if every $(PS)_d$ sequence of $E_A$ contains a convergent subsequence. 
	
		\begin{lemma}\label{PSbounded}
		Assume that $\{u_n\}$ is a $(PS)_d$ sequence of $E_A$ with $d>0$. Then $\{u_n\}$  is bounded in $\HA$.
	\end{lemma}
	\begin{proof}
		It holds that for $n$ large enough
		\begin{align*}
			d+1+o(1)\|u_n\|_A&\geq E_A(u_n)-\dfrac{1}{4}E_A^\prime(u_n)u_n\\
			&=\frac{1}{4}\Vert u_n\Vert_A^2+\left(\frac{1}{4}-\frac{1}{p}\right)\int_{\mathbb{R}^3}|u_n|^p\,dx\\
            &\geq\frac{1}{4}\Vert u_n\Vert_A^2,
		\end{align*}
where $o(1)\to 0$ as $n\to \infty$. This yields the boundedness of $\{u_n\}$ in $\HA$.
	\end{proof}
	
	Next, we establish the global compactness lemma, which is a relative of the one presented by Struwe \cite{S1984}, and which plays a significant role in recovering the compactness of $(PS)$ sequence.
	
	\begin{lemma}\label{splitting}
		Let $\{u_n\}\subset\HA$ be  a $(PS)_d$ sequence of the functional $E_A$ with $d>0$.
		Then there exists a solution $u^0\in\HA$ of \eqref{problem} such that one of the following cases occurs:
		\begin{enumerate}[{\rm (a)}]
			\item $u_n\rightarrow u^0$  in $\HA$ (up to a subsequence); 
			\item there exists a number $k\in\mathbb{N}^+$, $k$ sequences of points $\{y_n^j\}\subset\R$ with $|y_n^j|\rightarrow\infty$ and $|y_n^j-y_n^i|\rightarrow\infty$  for $j\neq i$, $1\leq j\leq k$, and $k$ nontrivial solutions $u^1,u^2,\cdots,u^k\in\H$ of \eqref{limitproblem}
			such that
			\begin{equation*}
				u_n-\sum_{j=1}^{k}u^j(\cdot-y_n^j)\rightarrow u^0~ {\rm in}~\H \quad \mbox{(up to a subsequence)}
			\end{equation*}
			and
			\begin{equation*}
			d=	E_A(u^0)+\sum_{j=1}^{k}E_0(u^j).
			\end{equation*}
		\end{enumerate}
	\end{lemma}
\begin{proof}
	By Lemma \ref{PSbounded}, we know that $\{u_n\}$  is bounded in $\HA$. Hence, there exists $u^0$ such that (up to a subsequence)
			\begin{equation*}
		\left\{
		\begin{aligned}
			&u_n\rightharpoonup u^0~{\rm in}~\HA,\\~
			&u_n\rightarrow u^0~{\rm in}~L^q_{loc}(\R,\mathbb{C}),\\
			&u_n\rightarrow u^0~{\rm a.e.~on}~\R,
		\end{aligned}
		\right.
	\end{equation*}
for every $1\leq q<6$. Since
    $E_A^\prime(u_n)\rightarrow0$, by Lemma \ref{propertyofphi} (c) we obtain
	\begin{align*}
		\Re\left(\int_{\mathbb{R}^3}(\nabla_Au^0\cdot\overline{\nabla_A\varphi}+V(x)u^0\overline{\varphi})dx\right)&+\Re\left(\int_{\mathbb{R}^3}K(x)\phi_{|u^0|}u^0\overline{\varphi}\,dx\right)\\
        &-\Re\left(\int_{\mathbb{R}^3}|u^0|^{p-2}u^0\overline{\varphi}dx\right)=0,\quad\forall\varphi\in \HA,
	\end{align*}
	which implies that $E_A^\prime(u^0)=0$. So $u^0$ is a solution of \eqref{problem}.
	
	Set $u_n^1:=u_n-u^0$. Since $u_n\rightharpoonup u^0$, we have
	\begin{equation*}
		\begin{aligned}
			\I |\nabla_A u_n^1|^2dx
			&=	\I |\nabla_A u_n-\nabla_A u^0|^2dx\\
&=		\I \left(|\nabla_A u_n|^2+|\nabla_A u^0|^2-2\Re\left(\nabla_A u_n\cdot\overline{\nabla_A u^0}\right)	\right)dx	\\		
		&=\I |\nabla_A u_n|^2dx-\I |\nabla_A u^0|^2dx+o(1)
		\end{aligned}
	\end{equation*}
for some $o(1)$ going to 0 as $n\to \infty$. In addition,
	\begin{equation}\label{gradient1}
		\begin{aligned}
		\I |\nabla u_n^1|^2dx&=\I |-i\nabla u_n^1|^2dx\\
		&=\I |-i\nabla u_n^1+\left(A(x)-A(x)\right)u_n^1|^2dx\\
		&=\I |\nabla_A  u_n^1-A(x)u_n^1|^2dx.
		\end{aligned}
		\end{equation}
Now,
\[
\I |\nabla_Au_n^1-A(x)u_n^1|^2\,dx=\I\left[|\nabla_Au_n^1|^2+|A(x)u_n^1|^2-2\Re \left(\nabla_Au_n^1\cdot\overline{A(x)u_n^1}\right)\right]\,dx.
\]
Since $A(x)\to 0 $ as $|x|\to \infty$, fixed $\epsilon>0$ there exists $R_\epsilon>0$ such that $|A(x)|<\epsilon$ if $|x|>R_\epsilon$. On the other hand, by the Rellich-Kondrachov theorem we have that $u_n^1\to 0$ in $L^2(B_{R_\epsilon})$, so that for $n$ large enough
\begin{equation}\label{azero}
\I|A(x)u_n^1|^2dx=\int_{B_{R_\epsilon}}|A(x)u_n^1|^2dx+\int_{\R\setminus B_{R_\epsilon}}|A(x)u_n^1|^2dx\leq \epsilon|A|_\infty^2+\epsilon^2|u_n^1|_2
\leq M\epsilon
\end{equation}
for some $M>0$, since ${u_n^1}$ is bounded in $L^2(\R,\C)$. On the other hand,
\[
\left|\I\Re \left(\nabla_Au_n^1\cdot\overline{A(x)u_n^1}\right)dx\right|\leq |\nabla_Au_n^1|_2|A(x)u_n^1|_2 \to 0
\]
as $n\to \infty$ from the previous lines. As a consequence, from \eqref{gradient1} and the weak convergence of $u_n$ to $u^0$, we find	
\begin{equation}\label{gradient2}
		\begin{aligned}		
\I |\nabla u_n^1|^2dx&=\I |\nabla_A  u_n^1|^2dx+o(1)\\
		&=\I |\nabla_A u_n|^2dx-\I |\nabla_A u^0|^2dx+o(1),
		\end{aligned}
		\end{equation}
where $o(1)\to 0$ as $n\to \infty$.

Moreover, using again the fact that $u_n\rightharpoonup u^0$,
\begin{equation}\label{Vx}
		\begin{aligned}
			&\quad\I \left(V_\infty|u_n^1|^2-V(x)|u_n|^2+V(x)|u^0|^2\right)dx\\
			&=\I \left[V_\infty\left(|u_n|^2-2\Re (u_n\overline{u^0})+|u^0|^2\right)-V(x)|u_n|^2+V(x)|u^0|^2\right]dx\\
			&=\I \left(V_\infty-V(x)\right)\left(|u_n|^2-|u^0|^2\right)dx+o(1)\\
			&=\I \left(V_\infty-V(x)\right)\left[(\overline{u_n}-\overline{u^0})u_n+(u_n-u^0)\overline{u^0}\right]dx+o(1)\\
			&=\I \left(V_\infty-V(x)\right)(\overline{u_n}-\overline{u^0})u_n\,dx-o(1),
		\end{aligned}
	\end{equation}
where $o(1)\to 0$ as $n\to \infty$. Reasoning as for the proof of \eqref{azero}, recalling $(V_1)$, from \eqref{Vx} we find that
\[
\I \left(V_\infty|u_n^1|^2-V(x)|u_n|^2+V(x)|u^0|^2\right)dx=o(1).
\]

Besides, using Br\'ezis-Lieb lemma \cite{BrezisLieb1983}, one has
	\begin{equation}\label{BrezisLieb}
		\I \left(|u_n|^p-|u_n^1|^p-|u^0|^p\right)dx \rightarrow 0
	\end{equation}
as $n\to \infty$. Hence, from \eqref{gradient2}-\eqref{BrezisLieb} and Lemma \ref{propertyofphi} (b), we conclude that
	\begin{align*}
	E_0(u_n^1)-E_A(u_n)+E_A(u^0)
	&=\dfrac{1}{2}\I \left(|\nabla u_n^1|^2-|\nabla_A u_n|^2+|\nabla_A u^0|^2\right)dx\\
    &-\frac{1}{4}\int_{\mathbb{R}^3}K(x)\left(\phi_{|u_n|}|u_n|^2-\phi_{|u^0|}|u^0|^2\right)\,dx
    \\
	&\quad+\dfrac{1}{2}\I \left(V_\infty|u_n^1|^2-V(x)|u_n|^2+V(x)|u^0|^2\right)dx\\
	&\quad-\dfrac{1}{p}\I \left(|u_n^1|^p-|u_n|^p+|u^0|^p\right)dx\\
	&=o(1).
\end{align*}
Thus
\begin{equation}\label{convspos}
E_0(u_n^1)\rightarrow d-E_A(u^0) \mbox{ as }n\to \infty.
\end{equation}

By virtue of \cite[Lemma 8.1]{Minimax}\footnote{Actually we need a version suitable for this setting, but the related proof is essentially the same of  that one, considering the real part of a complex-valued function.}, for any $\varphi\in \H$ we have
	\begin{equation}\label{Minimax}
	\Re	\left(\I |u_n|^{p-2}u_n\overline{\varphi}dx-\I|u_n^1|^{p-2}u_n^1\overline{\varphi}dx -\I |u^0|^{p-2}u^0\overline{\varphi}dx\right)=o(1)\Vert\varphi\Vert.
	\end{equation}
    Since $u_n\rightharpoonup u^0$ in $\HA$, reasoning as for the proof of \eqref{azero}, we have
    \begin{equation*}
    \Re\left(\I K(x)\phi_{|u_n|}\overline{\varphi}\left(u_n-u^0\right)dx\right)=o(1)\Vert\varphi\Vert.
    \end{equation*}
    Besides, \eqref{K(x)} yields that
    \begin{equation*}
      \Re\left(\int_{\mathbb{R}^3}K(x)\left(\phi_{|u_n|}-\phi_{|u|}\right)u^0\overline{\varphi}\,dx\right)=o(1)\Vert\varphi\Vert.  
    \end{equation*}
    As a result,
    \begin{equation}\label{K(x)un}
    \begin{aligned}
&\Re\left(\int_{\mathbb{R}^3}K(x)\left(\phi_{|u_n|}u_n\overline{\varphi}-\phi_{|u^0|}u^0\overline{\varphi}\right)\,dx\right)\\
=&\Re\left(\I K(x)\phi_{|u_n|}\overline{\varphi}\left(u_n-u^0\right)dx\right)+\Re\left(\I K(x)u^0\overline{\varphi}\left(\phi_{|u_n|}-\phi_{|u^0|}\right)\,dx\right)\\
=&o(1)\Vert\varphi\Vert.
    \end{aligned}
    \end{equation}

Let us consider the operator $\mathcal{J}:H^1(\R,\mathbb C)\to L^2(\R,\mathbb C^3)$
\begin{equation*}
    \mathcal{J}(\varphi):=A(x)\varphi .
\end{equation*}
By $(A_1)$, one can easily derive that the operator $\mathcal{J}$ is compact, reasoning as in \eqref{azero}. Hence its Hilbert adjoint
\begin{equation*}
    \mathcal{J}^*:L^2(\R,\mathbb C^3)\to H^{-1}(\R,\mathbb C)
\end{equation*}
is also compact. Since $u_n^1\rightharpoonup0$ in $H^1(\R,\mathbb C)$, we have
\begin{equation*}
    -i\nabla u_n^1\rightharpoonup0
\quad\text{in }L^2(\R,\mathbb C^3).
\end{equation*}
Therefore,
\begin{equation*}
    \mathcal{J}^*(-i\nabla u_n^1)\to0
\quad\text{in }H^{-1}(\R,\mathbb C).
\end{equation*}
Consequently,
\begin{equation*}
    \begin{aligned}
\operatorname{Re}\left(\int_{\mathbb R^3}
(-i\nabla u_n^1)\cdot A(x)\overline{\varphi}\,dx\right)
&=\operatorname{Re}\left\langle-i\nabla u_n^1,\mathcal{J}({\varphi})\right\rangle_{L^2(\R,\mathbb{C}^3)}
\\&=\operatorname{Re}
\left\langle \mathcal{J}^*(-i\nabla u_n^1),{\varphi}\right\rangle_{H^{-1}(\R,\mathbb{C}),H^1(\R,\mathbb{C})}\\
&\leq
\Vert\mathcal{J}^*(-i\nabla u_n^1)\Vert_{H^{-1}(\R,\mathbb{C})} \Vert\varphi\Vert=
o(1)\Vert\varphi\Vert.
\end{aligned}
\end{equation*}
Now, reasoning as for \eqref{azero}, by $(A_1)$ and $(V_1)$ we have that 
$$\Re\left(i\I A(x)u_n^1\cdot \overline{\nabla \varphi}dx\right)=o(1)\|\varphi\|,\quad \Re\left(\I |A(x)|^2
   u_n^1 \overline{ \varphi}\,dx\right)=o(1)\|\varphi\|$$
and
\begin{equation}\label{V(x)un}
    \Re\left(		
			\I \left(V(x)-V_\infty\right)u_n^1\overline{\varphi}\,dx\right)=o(1)\|\varphi\|.
\end{equation}
 From above,
\begin{equation}\label{A(x)un}
\begin{aligned}
   &\quad\Re\left( \I \nabla_A u_n^1\cdot \overline{\nabla_A \varphi}dx-\I \nabla u_n^1\cdot \overline{\nabla \varphi}dx\right)\\
   &= \Re\left(\I (-i\nabla u_n^1)\cdot A(x)\overline{ \varphi}\,dx+\I |A(x)|^2
   u_n^1 \overline{ \varphi}\,dx+i\I A(x)u_n^1\cdot \overline{\nabla \varphi}dx\right)\\
   &=o(1)\|\varphi\|.
\end{aligned}
\end{equation}
Note that 
if $\varphi\in H^1(\R,\mathbb{C})$, then by $(A_1)$  it is easy to check that $\varphi\in H_A^1(\R,\mathbb{C})$.
	From \eqref{Minimax}--\eqref{A(x)un}, we have
	\begin{equation*}
	\begin{aligned}
			E_A^\prime(u_n)\varphi
			&=	\Re\left(\int_{\mathbb{R}^3}(\nabla_Au_n\cdot\overline{\nabla_A\varphi}+V(x)u_n\overline{\varphi})dx\right)+\Re\left(\int_{\mathbb{R}^3}K(x)\phi_{|u_n|}u_n\overline{\varphi}\,dx\right)\\
            &\quad-\Re\left(\int_{\mathbb{R}^3}|u_n|^{p-2}u_n\overline{\varphi}dx\right)\\
            &\mbox{ (since $u_n=u_n^1+u^0$ we find)}\\
			&=E_A^\prime(u^0)\varphi+\Re\left(\I (\nabla_Au_n^1\cdot \overline{\nabla_A\varphi}+ V(x)u_n^1\overline{\varphi})dx\right)\\
            &\quad+\Re\left(\int_{\mathbb{R}^3}K(x)\left(\phi_{|u_n|}u_n\overline{\varphi}-\phi_{|u^0|}u^0\overline{\varphi}\right)\,dx\right)
            \\
			&\quad-\Re\left(\I|u_n|^{p-2}u_n\overline{\varphi}dx-\I|u^0|^{p-2}u^0\overline{\varphi}dx\right)\\
			 &\mbox{ (since $E'_A(u^0)=0$, adding and subtracting $E'_0(u_n^1)\varphi$, we get)}\\
			&=E_0^\prime(u_n^1)\varphi+\Re\left(\I (\nabla_Au_n^1\cdot \overline{\nabla_A\varphi}-\nabla u_n^1\cdot \overline{\nabla\varphi} )dx\right)+\Re\left(		
			\I \left(V(x)-V_\infty\right)u_n^1\overline{\varphi}\,dx\right)\\
            &\quad+\Re\left(\int_{\mathbb{R}^3}K(x)\left(\phi_{|u_n|}u_n\overline{\varphi}-\phi_{|u^0|}u^0\overline{\varphi}\right)\,dx\right)\\
			&\quad-\Re	\left(\I |u_n|^{p-2}u_n\overline{\varphi}dx-\I|u_n^1|^{p-2}u_n^1\overline{\varphi}dx -\I |u^0|^{p-2}u^0\overline{\varphi}dx\right)\\
&=E_0^\prime(u_n^1)\varphi+o(1)\|\varphi\|.
		\end{aligned}
	\end{equation*}
Then one has that
\[
|E_0^\prime(u_n^1)\varphi|\leq \|E_A^\prime(u_n)\|\|\varphi\|+o(1)\|\varphi\|=o(1)\|\varphi\|.
\]
Thus, we deduce that $E_0^\prime(u_n^1)\rightarrow0$. Moreover, from \eqref{convspos}, we get that  $\{u_n^1\}$ is a $(PS)$ sequence of $E_0$ at the level $d-E_A(u^0)$.
	
		Now we consider
	\begin{equation*}
		\delta_1:=\limsup_{n\to\infty}\sup_{y\in\R}\int_{B_1(y)}| u_n^1|^2dx.
	\end{equation*}
	Due to the diamagnetic inequality \eqref{diamagnetic}, $\{|u_n^1|\}$ is bounded in $\HR$. 
	If $\delta_1=0$, we know from Lions' concentration-compactness lemma \cite[Lemma I.1]{Lions2} that 
		\begin{equation}\label{lions}
		|u_n^1|\rightarrow 0~{\rm in}~L^q(\R,\mathbb{R}),\quad
        2<q<6.
	\end{equation}
Again as for \eqref{azero}, we have that 
	\begin{equation}\label{serve}
		\I |A(x)u_n^1|^2dx\rightarrow0\quad\mathrm{and}\quad 	\I (V_\infty-V(x))|u_n^1|^2dx\rightarrow0.
	\end{equation}
	By using $E_0^\prime(u_n^1)\rightarrow 0$, we have
\begin{equation}\label{sopra}
\begin{aligned}
	o(1)=	E_0^\prime(u_n^1)u_n^1&=\I \left(|\nabla u_n^1|^2+V_\infty |u_n^1|^2\right)dx-\I |u_n^1|^pdx\\
	&=\I \left(|-i\nabla u_n^1|^2+V(x) |u_n^1|^2\right)dx+\I (V_\infty-V(x))|u_n^1|^2dx-\I |u_n^1|^pdx\\
		&=\I \left(|-i\nabla u_n^1+(A(x)-A(x))u_n^1|^2+V(x)|u_n^1|^2\right)dx\\
		&\quad
		+\I (V_\infty-V(x))|u_n^1|^2dx-\I |u_n^1|^pdx\\
			&=\I \left(|\nabla_A u_n^1-A(x)u_n^1|^2+V(x)|u_n^1|^2\right)dx\\
			&\quad+\I (V_\infty-V(x))|u_n^1|^2dx-\I |u_n^1|^pdx\\
	&= \I |\nabla_Au_n^1|^2dx+\I |A(x)u_n^1|^2dx-2\Re \left(\I \nabla_Au_n^1\cdot \overline{A(x)u_n^1}dx  \right)\\
	&\quad+\I (V_\infty-V(x))|u_n^1|^2dx-\I |u_n^1|^pdx
	\end{aligned}
\end{equation}
Noting that
\[
\left|\I \nabla_Au_n^1\cdot \overline{A(x)u_n^1}dx \right|\leq |\nabla_Au_n^1|_2|A(x)u_n^1|_2,
\]
and recalling that $A$ is bounded, from \eqref{lions}, \eqref{serve} and \eqref{sopra} we find that
\[
E_0^\prime(u_n^1)u_n^1=\Vert u_n^1\Vert_A^2-\I |u_n^1|^pdx+o(1),
\]
for some $o(1)\to 0$, which implies that $u_n^1\rightarrow0$ in $\HA$. Thus, $u_n\to u^0$ in $\HA$ and the proof is complete.

{\sl Viceversa}, if $\delta_1>0$, there exists a sequence $\{y_n^1\}\subset\R$ such that
	\begin{equation}\label{yn1}
		\int_{B_1(y_n^1)}|u_n^1|^2dx>\dfrac{\delta_1}{2} \mbox{ for every }n\in \N.
	\end{equation}
Now set $v_n^1:=u_n^1(\cdot+y_n^1) $. Since $\{u_n^1\} $ is bounded in $\H$, then $ \{v_n^1\}$ is also bounded in $\H$. Hence, there exists $u^1\in \H$ such that 
	\begin{equation*}
		\left\{
		\begin{aligned}
			&v_n^1\rightharpoonup u^1~{\rm in}~\H,\\~
			&v_n^1\rightarrow u^1~{\rm in}~L^q_{loc}(\R,\mathbb{C}),\\
			&v_n^1\rightarrow u^1~{\rm a.e.~on}~\R,
		\end{aligned}
		\right.
	\end{equation*}
for any $ q\in [1,6)$.
	From (\ref{yn1}), it follows that
	\begin{equation*}
		\int_{B_1(0)}|v_n^1|^2dx>\dfrac{\delta_1}{2}
	\end{equation*}
and by  the strong convergence in $B_1(0)$, we get
	\begin{equation}\label{nontriv}
		\int_{B_1(0)}|u^1|^2dx\geq\dfrac{\delta_1}{2},
	\end{equation}
and so $u^1\neq0$.
	 Since $u_n^1\rightharpoonup0$, we deduce that $\{y_n^1\}$ is unbounded and, without loss of generality, we can assume that $|y_n^1|\rightarrow\infty$.
	 
Now, take $\varphi\in C^\infty_c(\R,\C)$. Then
\[
E_0'(v_n^1)\varphi=E_0^\prime(u_n^1(\cdot+y_n^1))\varphi=E_0^\prime(u_n^1)\varphi(\cdot-y_n^1)=o(1),
\]
since $E_0'(u_n^1)\to 0$. But using the weak convergence in $\H$ and the strong convergence in the support of $\varphi$, we find
\[
\begin{aligned}
o(1)=E_0'(v_n^1)\varphi&=\int_{\R}\nabla v_n^1\cdot \nabla \varphi\,dx+\int_{\R}V_\infty v_n^1\varphi\,dx-\int_{\R}|v_n^1|^{p-2}v_n^1\varphi\,dx\\
&\to \int_{\R}\nabla u^1\cdot \nabla \varphi\,dx+\int_{\R}V_\infty u^1\varphi\,dx-\int_{\R}|u^1|^{p-2}u^1\varphi\,dx=E_0'(u^1)\varphi,
\end{aligned}
\]
that is $E_0^\prime(u^1)\varphi=0$ for every $\varphi\in C^\infty_c(\R,\C)$. By density, we finally get that
 $E_0^\prime(u^1)=0$ and so,  by \eqref{nontriv}, $u^1$ is a nontrivial solution of \eqref{limitproblem}.

Now, set $u_n^2:= u_n^1-u^1(\cdot-y_n^1)=u_n-u^0-u^1(\cdot-y_n^1)$.
	Using similar arguments, we can show that $\{u_n^2\}$ is a $(PS)$ sequence of $E_A$ at the level $d-E_A(u^0)-E_0(u^1)$.
	 
	 Define
	 \begin{equation*}
	 	\delta_2:=\limsup_{n\to\infty}\sup_{y\in\R}\int_{B_1(y)}| u_n^2|^2dx.
	 \end{equation*}
	Similarly to what we proved before, if $\delta_2=0$, then we have that $u_n^2\rightarrow0$ in $\H$. Hence, $u_n\to u^0+u^1(\cdot-y_n^1)$ in $\H$. Besides, one can easily verify that $E_0(u_n^1(\cdot+y_n^1))=E_0(u_n^1)$, so $E_0(u_n^1)\rightarrow E_0(u^1)=d-E_A(u^0)$. As a result,
	\begin{equation*}
		d=E_A(u^0)+E_0(u^1).
	\end{equation*}
Then the proof is completed. 

If $\delta_2>0$, we proceed as before, obtaining  a nontrivial solution $u^2$ of \eqref{limitproblem} and a sequence $\{y_n^2\}\subset\R$ satisfying $|y_n^2|\rightarrow\infty$ and $|y_n^2-y_n^1|\rightarrow\infty$.
	
We claim that the above procedure can stop only after finitely many steps. Indeed, if $u$ is a nontrivial solution of  \eqref{limitproblem}, then
\begin{equation*}
	E_0(u)\geq m_0>0.
\end{equation*}
Since $d$ is a finite number, the procedure must terminate after a finite number of steps. Suppose that the procedure performs $k+1$ times, where $k\in\mathbb{N}^+$, that is
\begin{equation*}
	\delta_{k+1}:=\limsup_{n\to\infty}\sup_{y\in\R}\int_{B_1(y)}| u_n^{k+1}|^2dx=0.
\end{equation*}
Then
	\begin{equation*}
	u_n-\sum_{j=1}^{k}u^j(\cdot-y_n^j)\rightarrow u^0~ {\rm in}~\H
\end{equation*}
and
\begin{equation*}
	d=	E_A(u^0)+\sum_{j=1}^{k}E_0(u^j).
\end{equation*}
This completes the proof.
\end{proof}
	
\begin{corollary}\label{cpt1}
	The functional $E_A$ satisfies the $(PS)_d$ condition for every  $ d\in(0,m_0)$.
\end{corollary}	
\begin{proof}
	Let $\{u_n\}$ be a $(PS)_d$ sequence with $ d\in(0,m_0)$.
	By virtue of Lemma \ref{PSbounded}, $\{u_n\}$ is bounded in $\HA$. Hence, there exists $u^0\in\HA$ such that 
	\begin{equation*}
		u_n\rightharpoonup u^0 \quad\mathrm{in}~\HA.
	\end{equation*}
	Suppose that $u_n\not\rightarrow u^0$ in $\HA$. From Lemma \ref{splitting}, it follows that there exists $k\in\mathbb{N}^+$ and $k$ nontrivial solutions $u^1,u^2,\cdots,u^k\in\H$ of \eqref{limitproblem}
	such that	
	\begin{equation*}
		d=	E_A(u^0)+\sum_{j=1}^{k}E_0(u^j)\geq km_0,
	\end{equation*}
	which is a contradiction with $d\in(0,m_0)$. Thus, $u_n\rightarrow u^0$ in $\HA$.
\end{proof}	

Recall the definition
\begin{equation*}
\hat{m}=\min\{m_0,\xi\}
\end{equation*} 

\begin{corollary}\label{cpt2}
		If $m_A=m_0$, then the functional
		$E_A$ satisfies the $(PS)_d$ condition for every $d\in(m_0,m_0+\hat{m})$.
\end{corollary}		
\begin{proof}
		Let $\{u_n\}$ be a $(PS)_d$ sequence. Lemma \ref{PSbounded} yields that $\{u_n\}$ is bounded in $\HA$.
	Arguing by contradiction, suppose that (up to a subsequence) $u_n\rightharpoonup u^0$ in $\HA$ but $u_n\not\rightarrow u^0$ in $\HA$. We claim that $u^0=0$. If not, by Lemma \ref{splitting},  there exist $k\in\mathbb{N}^+$ and $k$ nontrivial solutions $u^1,u^2,\cdots,u^k\in\H$ of \eqref{limitproblem}
	such that	
	\begin{equation*}
		d=	E_A(u^0)+\sum_{j=1}^{k}E_0(u^j)\geq (k+1)m_0\geq2m_0,
	\end{equation*}
    which is impossible. Hence,
    \begin{equation*}
        d=\sum_{j=1}^{k}E_0(u^j)\geq km_0,
    \end{equation*}
	which implies that $k=1$. Then
	\begin{equation*}
			d=E_0(u^1)<m_0+\hat{m}<m_0+\xi.
	\end{equation*}
	By Lemma \ref{isolatedlevel}, we obtain $d=E_0(u^1)=m_0$. This is a contradiction. Hence, we conclude that  $u_n\rightarrow u^0$ in $\HA$.
\end{proof}

\section{Existence of a bound state solution}

This section is devoted to proving Theorem \ref{highenergy}. We begin by studying the relationship between the least energy $m_A$ and $m_0$.

\begin{lemma}\label{le:0<mAleqm0}
    Assume that $(A_1), (V_1), (K_1)$ hold. Then
    \begin{align*}
        0<m_A\leq m_0.
    \end{align*}
\end{lemma}
\begin{proof}
    Let $u\in\mathcal{N}_A$. From  Lemma \ref{properties} (a), it follows that
	\begin{align*}
		E_A(u)&=\dfrac{1}{2}\Vert u\Vert_A^2+\frac{1}{4}\int_{\mathbb{R}^3}K(x)\phi_{|u|}|u|^2\,dx-\dfrac{1}{p}\I|u|^p\,dx\\
        &=\frac{1}{4}\Vert u\Vert_A^2+\left(\frac{1}{4}-\frac{1}{p}\right)\int_{\mathbb{R}^3}|u|^p\,dx\\
		&\geq\frac{1}{4}\Vert u\Vert_A^2
		\geq\frac{1}{4}\rho^2>0,
	\end{align*}
	which implies that $m_A>0$. Next we prove $m_A\leq m_0$. It suffices to construct a sequence $\{u_n\}\subset\mathcal{N}_A$ such that $E_A(u_n)\to m_0$. Take a sequence $\{y_n\}$ with $|y_n|\to\infty$ as $n\to\infty$ and set $u_n:=t_n\omega_{y_n}=t_n\omega(\cdot-y_n)$, where $t_n$ is such that $t_n\omega_{y_n}\in\mathcal{N}_A$. Then we have
    \begin{equation}\label{11}
    \begin{aligned}
         E_A(u_n)&=\frac{t_n^2}{2}\Vert\omega_{y_n}\Vert_A^2+\frac{t_n^4}{4}\int_{\mathbb{R}^3}K(x)\phi_{|\omega_{y_n}|}|\omega_{y_n}|^2\,dx-\frac{t_n^p}{p}\int_{\mathbb{R}^3}|\omega_{y_n}|^p\,dx\\
         &=\frac{t_n^2}{2}\int_{\mathbb{R}^3}\left(|\nabla\omega|^2+\left(|A(x+y_n)|^2+V(x+y_n)\right)||\omega|^2\right)\,dx\\
         &\quad+\frac{t_n^4}{4}\int_{\mathbb{R}^3}K(x)\phi_{|\omega_{y_n}|}|\omega_{y_n}|^2\,dx-\frac{t_n^p}{p}\int_{\mathbb{R}^3}|\omega|^p\,dx.
    \end{aligned}
    \end{equation}
    Since $u_n\in\mathcal{N}_A$, one has
    \begin{equation}\label{22}
         \begin{aligned}
        0=E_A'(u_n)u_n&=t_n^2\left(\int_{\mathbb{R}^3}\left(|\nabla\omega|^2+\left(|A(x+y_n)|^2+V(x+y_n)\right)|\omega|^2\right)\,dx\right)\\
        &\quad +t_n^4\int_{\mathbb{R}^3}K(x)\phi_{|\omega_{y_n}|}|\omega_{y_n}|^2\,dx-t_n^p\int_{\mathbb{R}^3}|\omega|^p\,dx.
    \end{aligned}
    \end{equation}
     By $(A), (V)$ and the dominated convergence theorem,  we obtain
     \begin{align*}
         \int_{\mathbb{R}^3}\left(|A(x+y_n)|^2+V(x+y_n)-V_\infty\right)|\omega|^2\,dx\to0.
     \end{align*}
     Besides, since $\omega_{y_n}\rightharpoonup0$ in $\HA$, Lemma \ref{propertyofphi} (b) implies that
     \begin{align*}
         \int_{\mathbb{R}^3}K(x)\phi_{|\omega_{y_n}|}|\omega_{y_n}|^2\,dx\to0.
     \end{align*}
Then we have 
\begin{equation*}
    \begin{aligned}       \|\omega_{y_n}\|_A^2&=\int_{\R}\left(|\nabla\omega_{y_n}|^2+\left(|A(x)|^2+V(x)-V_\infty\right)|\omega_n|^2\right)dx+\int_{\R}V_\infty|w_{y_n}|^2dx\\
    &=\int_{\R}\left(|\nabla\omega|^2+\left(|A(x+y_n)|^2+V(x+y_n)-V_\infty\right)|\omega|^2\right)dx+\int_{\R}V_\infty|w|^2dx\\
    &=\int_{\R}\left(|\nabla\omega|^2+V_\infty|\omega|^2\right)dx+o(1).
    \end{aligned}
\end{equation*}
From \eqref{22}, it follows that
     \begin{equation*}
\int_{\R}\left(|\nabla\omega|^2+V_\infty|\omega|^2\right)dx+   t_n^2\int_{\mathbb{R}^3}K(x)\phi_{|\omega_{y_n}|}|\omega_{y_n}|^2dx = t_n^{p-2}   \int_{\mathbb{R}^3}|\omega|^pdx+o(1).
     \end{equation*}
Recalling that $\omega$ satisfies
\begin{equation*}
    \int_{\R}\left(|\nabla\omega|^2+V_\infty|\omega|^2\right)dx=\int_{\mathbb{R}^3}|\omega|^pdx,
\end{equation*}
 one has
\begin{equation*}
    \left(1-t_n^{p-2}\right)\int_{\R}\left(|\nabla\omega|^2+V_\infty|\omega|^2\right)dx+t_n^2\int_{\mathbb{R}^3}K(x)\phi_{|\omega_{y_n}|}|\omega_{y_n}|^2dx=o(1),
\end{equation*}
     which implies that $t_n\to1$ as $n\to\infty$.
Hence, as $n\to\infty$,
     \begin{align*}
         E_A(u_n)&=\frac{t_n^2}{2}\int_{\mathbb{R}^3}\left(|\nabla\omega|^2+\left(|A(x+y_n)|^2+V(x+y_n)\right)|\omega|^2\right)\,dx\\
         &\quad+\frac{t_n^4}{4}\int_{\mathbb{R}^3}K(x)\phi_{|\omega_{y_n}|}|\omega_{y_n}|^2\,dx-\frac{t_n^p}{p}\int_{\mathbb{R}^3}|\omega|^p\,dx\\
         &\to \frac{1}{2}\int_{\mathbb{R}^3}\left(|\nabla\omega|^2+V_\infty|\omega|^2\right)\,dx-\frac{1}{p}\int_{\mathbb{R}^3}|\omega|^p\,dx
         =E_0(\omega)=m_0.
     \end{align*}
     This completes the proof.
\end{proof}

	We recall the definition of barycenter $\beta$ of a function $u\in \HA\setminus \{0\}$ from \cite{ACR2008,CP2003}. Set
	\begin{equation*}
		\mu(u)(x):=\dfrac{1}{|B_1(x)|}\int_{B_1(x)}|u(y)|dy.
	\end{equation*}
	It is easy to check that $ \mu(u)\in L^\infty(\R)$ and that $\mu(u)(x)\to 0$ as $|x|\to \infty$. Moreover, $\mu(u)$ is continuous in $\R$, so that there exists $\max_{\R}\mu(u)$. Besides, the map  $\HA\ni u\mapsto \mu(u)\in L^\infty(\R)$ is evidently continuous. Next, we set
	\begin{equation*}
		\hat{u}(x):=\left[\mu(u)(x)-\dfrac{1}{2}\max_{\R}\mu(u)\right]^+.
	\end{equation*} 
	It is clear that  $\hat{u}$ is a nonnegative function and  has compact support.
This fact permits to define the barycenter function $\beta$ for $u$ by setting
	\begin{equation*}
		\beta(u):=\dfrac{1}{|\hat{u}|_1}\I x\hat{u}(x)dx\in \R.
	\end{equation*}
Since $\mu$ depends continuously on $u$,  $\beta$ is a continuous function in $ \HA\backslash\{0\}$. Moreover, we recall the following facts (see \cite{ACR2008}).
	
	\begin{lemma}\label{barycenterproperty}
		$\beta$ satisfies the following properties:
		\begin{enumerate}[{\rm (a)}]
			\item  If $ u $ is a radial function, then $ \beta(u)=0 $;
			\item  For any $t\neq0$, $u\in \HA\backslash\{0\} $, $\beta(tu)=\beta(u)$;
			\item For any $y\in\mathbb{R}^3$ and $u_y=u(\cdot-y)$, $ \beta(u_y)=\beta(u)+y $.
		\end{enumerate}
	\end{lemma}

	Define
\begin{equation*}
	b:=\inf\{E_A(u):u\in\mathcal{N}_A~{\rm and}~\beta(u)=0\}.
\end{equation*}
\begin{lemma}\label{b>m0}
	If $m_A$ is not achieved, then
    \begin{align*}
        m_A=m_0<b.
    \end{align*}
\end{lemma}
\begin{proof}
   First, we prove that if $m_A<m_0$, then $m_A$ is achieved, namely \eqref{problem} has a ground state solution. By definition of $m_A$, there exists a minimizing sequence $\{v_n\}\subset\mathcal{N}_A$ such that $E_A(v_n)\to m_A$. By virtue of the Ekeland variational principle, we find two sequences $\{u_n\}\subset\mathcal{N}_A$ and $\{\lambda_n\}\subset\mathbb{R}$ satisfying
   \begin{equation*}
	E_A(u_n)\rightarrow m_A,~E_A^\prime(u_n)-\lambda_n G^\prime_A(u_n)\rightarrow0~{\rm and}~\Vert u_n-v_n\Vert_A\rightarrow0~{\rm as}~n\rightarrow\infty,
\end{equation*}
where $G_A(u):=E_A^\prime(u)u$, see the proof of Lemma \ref{properties}.  Since $E_A(u_n)$ is bounded and $\{u_n\}\subset\mathcal{N}_A$, we have
\begin{equation*}
	\begin{aligned}
		m_A+o(1)=E_A(u_n)&=\dfrac{1}{4}\Vert u_n\Vert_A^2+\left(\dfrac{1}{4}-\dfrac{1}{p}\right)\I |u_n|^pdx
		>\dfrac{1}{4}\Vert u_n\Vert_A^2,
	\end{aligned}
\end{equation*}
which implies that $\{u_n\}$ is bounded in $\HA$. 
 For any $\varphi\in \HA$, using the H\"older inequality and the boundedness of $\{u_n\}$ in $\HA$, it is easy to check that
\begin{equation*}
	\begin{aligned}
		|G^\prime_A(u_n)\varphi|&=\left|\Re\left(2\I (\nabla_A u_n\cdot\overline{\nabla_A \varphi}+V(x)u_n\overline{\varphi})dx+4\int_{\R}K(x)\phi_{|u_n|}u_n\overline{\varphi}dx-p\I |u_n|^{p-2}u_n\overline{\varphi}dx\right)\right|\\
        &\leq C\Vert\varphi\Vert_A,
	\end{aligned}
\end{equation*}
which implies that $\left\{G^\prime_A(u_n)\right\}$ is bounded in $(\HA)^*$.  By $E_A^\prime(u_n)-\lambda_n G^\prime_A(u_n)\rightarrow0$,
 we get
\begin{equation*}
	E_A^\prime(u_n)u_n-\lambda_n G^\prime_A(u_n)u_n\rightarrow0.
\end{equation*}
It follows from the proof of Lemma \ref{properties} (b) that $\lambda_n\rightarrow0$. Thus, we deduce that $\{u_n\}$ is a $(PS)_{m_A}$ sequence of $E_A$. Since $0<m_A<m_0$, Corollary \ref{cpt1} implies that $\{u_n\}$ satisfies the $(PS)_{m_A}$ condition. As a result, we obtain a nontrivial critical point $u\in \HA$ of $E_A$ with $E_A(u)=m_A$.

Next, we show that $b>m_A=m_0$ provided that $m_A$ is not achieved. It is clear that $b\geq m_A$. We suppose by contradiction that $b=m_A$.
	By definition of $b$, there exists a minimizing sequence $\{\hat{v_n}\}\subset\mathcal{N}_A$ satisfying $E_A(\hat{v_n})\rightarrow b=m_A$ and $\beta(\hat{v_n})=0$ for every $n\in \N$. Again, using the Ekeland variational principle,  we find two sequences $\{\hat{u_n}\}\subset\mathcal{N}_A$ and $\{\lambda_n\}\subset\mathbb{R}$ satisfying
	\begin{equation*}
		E_A(\hat{u_n})\rightarrow m_A,~E_A^\prime(\hat{u_n})-\lambda_n G^\prime_A(\hat{u_n})\rightarrow0~{\rm and}~\Vert \hat{u_n}-\hat{v_n}\Vert_A\rightarrow0~{\rm as}~n\rightarrow\infty.
	\end{equation*}
	By a similar argument as above, we prove that $\{\hat{u_n}\}$ is a $(PS)_{m_A}$ sequence of $E_A$. As $m_A$ is not achieved, combining Lemma \ref{splitting} and the fact $m_A=m_0$,
	we get a sequence $\{y_n\}$ satisfying $|y_n|\rightarrow\infty$ and a nontrivial solution $u^1$ of \eqref{limitproblem} such that
	\begin{equation*}
		m_A=m_0=E_0(u^1)\quad\mathrm{and}\quad \hat{u_n}-u^1(\cdot-y_n)\to0.
	\end{equation*}
From Lemma \ref{isolatedlevel}, we deduce that $u^1=\omega$. 
Since $\omega$ is a radial function, Lemma \ref{barycenterproperty} (a) yields that $\beta(\omega)=0$. Set $z_n:=\hat{u_n}(\cdot+y_n)$, then $z_n\to u^1=\omega$. Applying the barycenter mapping to
both sides, we obtain
\begin{equation*}
	-y_n=\beta(\hat{u_n})-y_n=\beta(z_n)\rightarrow\beta(\omega)=0,
\end{equation*}
which contradicts to the fact that $|y_n|\rightarrow\infty$. Hence, we conclude that $b>m_A=m_0$.
\end{proof}	
	
	Next, we introduce the operator $\mathcal{T}:\R\rightarrow\mathcal{N}_A$ defined by
	\begin{equation*}
		\mathcal{T}[y](x):=t_y\omega_y=t_y\omega(x-y),
	\end{equation*}
	where $ t_y $ projects $ \omega(\cdot-y) $ onto $ \mathcal{N}_A $, according to Lemma \ref{properties} (d). 
	
	\begin{lemma}\label{beta1}
		$\beta\left(\mathcal{T}[y]\right)=y$.
	\end{lemma}
\begin{proof}
		Note that $\omega$ is radial. From Lemma \ref{barycenterproperty} (a) and (c), it follows that
	\begin{equation*}
		\beta\left(\mathcal{T}[y]\right)=\beta\left(t_y\omega(\cdot-y)\right)=\beta\left(\omega(\cdot-y)\right)=\beta\left(\omega\right)+y=y,
	\end{equation*}
	which concludes the proof.
\end{proof}
	
	\begin{lemma}\label{beta2}
		$E_A\left(\mathcal{T}[y]\right)\rightarrow m_0$ if $|y|\rightarrow\infty$.
	\end{lemma}
\begin{proof}
	By the proof of Lemma \ref{le:0<mAleqm0}, we have $t_y\rightarrow1$ as $|y|\rightarrow\infty$.
	Moreover,
		\begin{equation*}
		\begin{aligned}
			E_A\left(\mathcal{T}[y]\right)=E_A(t_y\omega_y)&=\dfrac{t_y^2}{2}\Vert \omega_y\Vert_A^2+\frac{t_y^4}{4}\int_{\mathbb{R}^3}K(x)\phi_{|\omega_y|}|\omega_y|^2\,dx-\dfrac{t_y^p}{p}\I |\omega_y|^pdx\\
			&=\dfrac{t_y^2}{2}\I \left(|\nabla \omega|^2+\left(|A(x+y)|^2+V(x+y)\right)|\omega|^2\right)dx\\
            &\quad+\frac{t_y^4}{4}\int_{\mathbb{R}^3}K(x)\phi_{|\omega_y|}|\omega_y|^2\,dx-\dfrac{t_y^p}{p}\I |\omega|^pdx\\
			&=\dfrac{t_y^2}{2}\I \left(|\nabla \omega|^2+V_\infty|\omega|^2\right)dx-\dfrac{t_y^p}{p}\I |\omega|^pdx\\
			&\quad+\dfrac{t_y^2}{2}\I \left(|A(x+y)|^2+V(x+y)-V_\infty\right)|\omega|^2dx\\
            &\quad+\frac{t_y^4}{4}\int_{\mathbb{R}^3}K(x)\phi_{|\omega_y|}|\omega_y|^2\,dx\\
			&\rightarrow E_0(\omega)=m_0.
		\end{aligned}
	\end{equation*}
	This completes the proof.
\end{proof}

	\begin{lemma}\label{beta3}
		Under the assumptions of  Theorem $\ref{highenergy}$, then for  $y\in\R$,
		\begin{equation*}
		E_A\left(\mathcal{T}[y]\right)<m_0+\hat{m},
		\end{equation*}
			where $\hat{m}=\min\{m_0,\xi\}$ and $\xi$ is given in Lemma $\ref{isolatedlevel}$.
	\end{lemma}
    \begin{proof}
        By Lemma \ref{properties} (d), there exists $t_y>0$ such that $t_y\omega_y\in\mathcal{N}_A$. Then
        \begin{equation}\label{da+1}
            t_y^2\Vert\omega_y\Vert_A^2+t_y^4\int_{\mathbb{R}^3}K(x)\phi_{|\omega_y|}|\omega_y|^2\,dx-t_y^p\int_{\mathbb{R}^3}|\omega_y|^p\,dx=0.
        \end{equation}

Since $\omega\in\mathcal{N}_0$, one has
        \begin{align*}
            1&=\frac{\Vert\omega_y\Vert^2_\bullet}{\int_{\mathbb{R}^3}|\omega_y|^p\,dx}\\
            &=\frac{t_y^{p-2}\int_{\mathbb{R}^3}|\omega_y|^p\,dx-t_y^2\int_{\mathbb{R}^3}K(x)\phi_{|\omega_y|}|\omega_y|^2\,dx-\int_{\mathbb{R}^3}\left(|A(x)|^2+V(x)-V_\infty\right)|\omega_y|^2\,dx}{\int_{\mathbb{R}^3}|\omega_y|^p\,dx}\\
            &\leq t_y^{p-2},
        \end{align*}
        which implies that $t_y\geq1$.
        As a result, from \eqref{da+1} we have
        \begin{align*}
            t_y^4\left(\Vert\omega_y\Vert_A^2+\int_{\mathbb{R}^3}K(x)\phi_{|\omega_y|}|\omega_y|^2\,dx-t_y^{p-4}\int_{\mathbb{R}^3}|\omega_y|^p\,dx\right)\geq0.
        \end{align*}
By \eqref{normphi} and \eqref{S-1}, we have
        \begin{equation*}
            \int_{\mathbb{R}^3}K(x)\phi_{|\omega_y|}|\omega_y|^2\,dx\leq S^{-3}|K|_2^2\Vert \omega_y\Vert^4_{\mathcal{D}^{1,2}}= S^{-3}|K|_2^2\Vert \omega\Vert^4_{\mathcal{D}^{1,2}}\leq S^{-3}|K|_2^2\Vert \omega\Vert^4_A.
        \end{equation*}
        From \eqref{S-1}, recalling that $\|\omega\|_\bullet^2=|\omega|_p^p$, it follows that
        \begin{equation}\label{t_y}
            \begin{aligned}
                 t_y^{p-4}&\leq\frac{\Vert\omega_y\Vert_A^2+\int_{\mathbb{R}^3}K(x)\phi_{|\omega_y|}|\omega_y|^2\,dx}{\int_{\mathbb{R}^3}|\omega_y|^p\,dx}\\
            &\leq\frac{\Vert\omega\Vert^2_\bullet+\int_{\mathbb{R}^3}\left(|A(x)|^2+V(x)-V_\infty\right)|\omega_y|^2\,dx+\int_{\mathbb{R}^3}K(x)\phi_{|\omega_y|}|\omega_y|^2\,dx}{\Vert\omega\Vert^2_\bullet}\\
            &< 1+ V_\infty ^{-1}\sup_{x\in\mathbb{R}^3}\left(|A(x)|^2+V(x)-V_\infty\right)+\frac{S^{-3}\left|K\right|_2^2\Vert\omega\Vert_A^4}{\Vert\omega\Vert^2_\bullet}.
            \end{aligned}
\end{equation}

By Remark \ref{Re.N0property}, since $\omega\in\mathcal{N}_0$, then
\begin{equation*} \max_{t>0}E_0(t\omega)=E_0(\omega)=m_0.
\end{equation*}
Using the fact that $t_y\geq 1$, from \eqref{AE} and \eqref{t_y} we conclude that
        \begin{align*}
			E_A\left(\mathcal{T}[y]\right)
			&=\dfrac{t_y^2}{2}\I \left(|\nabla \omega|^2+V_\infty|\omega|^2\right)dx-\dfrac{t_y^p}{p}\I |\omega|^pdx\\
			&\quad+\dfrac{t_y^2}{2}\I \left(|A(x+y)|^2+V(x+y)-V_\infty\right)|\omega|^2dx+\frac{t_y^4}{4}\int_{\mathbb{R}^3}K(x)\phi_{|\omega_y|}|\omega_y|^2\,dx
            \\
            &=E_0(t_y\omega)+\dfrac{t_y^2}{2}\I \left(|A(x+y)|^2+V(x+y)-V_\infty\right)|\omega|^2dx+\frac{t_y^4}{4}\int_{\mathbb{R}^3}K(x)\phi_{|\omega_y|}|\omega_y|^2\,dx
            \\
			&\leq m_0+\dfrac{t_y^2}{2}\I \left(|A(x+y)|^2+V(x+y)-V_\infty\right)|\omega|^2dx+\frac{t_y^4}{4}\int_{\mathbb{R}^3}K(x)\phi_{|\omega_y|}|\omega_y|^2\,dx\\
            &\leq m_0+\frac{t_y^4}{2}\left(\sup_{x\in\mathbb{R}^3}(|A(x)|^2+V(x)-V_\infty)|w|_2^2+S^{-3}|K|_2^2\Vert \omega \Vert_A^4 \right)\\
            &<m_0+\hat{m},
		\end{align*}
        as desired.
    \end{proof}

\begin{proof}[{\rm \textbf{Proof of Theorem \ref{highenergy}}}]
By Lemma \ref{le:0<mAleqm0}, we have $m_A\leq m_0$. If $m_A$ is achieved, the proof is completed. If $m_A$ is not achieved, Lemma \ref{b>m0} yields that $m_A=m_0$.
In the following, we prove the existence of a higher-energy bound state solution.
From  Lemma \ref{b>m0}, we know that $m_A=m_0<b$. Then there exists a positive number $\delta>0$ such that $b>m_0+\delta$.  Lemma \ref{beta1} implies that $\beta\left(\mathcal{T}[0]\right)=0$, so that $E_A\left(\mathcal{T}[0]\right)\geq b$. Besides, from Lemma \ref{beta3}, one has $E_A\left(\mathcal{T}[0]\right)<m_0+\hat{m}$, so there exists a positive number $\tau>0$ such that $b\leq m_0+\hat{m}-\tau$. Hence, we conclude that
\begin{align*}
    m_A=m_0<m_0+\delta<b\leq m_0+\hat{m}-\tau<m_0+\hat{m}.
\end{align*}
Arguing by contradiction, we assume that $E_A$ has no critical values in $(m_0, m_0+\hat{m})$. It follows from Corollary \ref{cpt2} and the deformation lemma \cite[Lemma 5.15]{Minimax} that there is a continuous retraction given by
\begin{align*}
    \chi:\mathcal{N}_A\cap E_A^{m_0+\hat{m}-\tau}\to \mathcal{N}_A\cap E_A^{m_0+\delta},
\end{align*}
where $E_A^a=\left\{u\in\HA:E_A(u)\leq a\right\}$. It holds that $\chi(u)=u$ for all $u\in\mathcal{N}_A\cap E_A^{m_0+\delta}$. Using Lemma \ref{beta2}, we deduce that for $\delta$ above, there exists $R>0$ such that 
\begin{equation}\label{|y|=R}
	m_0<\max_{|y|=R}E_A\left(\mathcal{T}[y]\right)<m_0+\delta<b.
\end{equation}
Next, we define a map $\mathcal{H}:\overline{B_R(0)}\to \partial B_R(0)$ by
\begin{equation*}
	\mathcal{H}(y):=R\cdot\frac{\beta\circ\chi\circ\mathcal{T}[y]}{\left|\beta\circ\chi\circ\mathcal{T}[y]\right|}.
\end{equation*}
One can easily check that $\mathcal{H}$ is continuous. Furthermore, Lemma \ref{beta1} and \eqref{|y|=R} imply that $\mathcal{H}(y)=y$ for all $y\in\partial B_R(0)$, so we conclude that $\mathcal{H}$ is a continuous retraction from $\overline{B_R(0)}$ onto $\partial B_R(0)$. Let us consider the homotopy:
\begin{equation*}
    \mathcal{Q}:[0,1]\times {\overline{B_R(0)}}\to \mathbb{R}^3,\quad \mathcal{Q}(t,u)=(1-t)y+t\mathcal{H}(y).
\end{equation*}
For any $y\in\partial B_R(0)$, we have
\begin{equation*}
	\mathcal{Q}(t,u)=(1-t)y+ty=y\neq0,
\end{equation*}
which yields that $0\not\in \mathcal{Q}\left([0,1]\times{\partial B_R(0)}\right)$. Moreover,  by virtue of the homotopy invariance
of the topological degree, one has
\begin{equation*}
    {\rm deg}\left(\mathcal{H}, B_R(0), 0\right)={\rm deg}\left(id, B_R(0), 0\right)=1,
\end{equation*}
which is a contradiction. Therefore, we conclude that $E_A$ has a critical value in $(m_0, m_0+\hat{m})$, namely \eqref{problem} has a bound state solution $u$ with $E_A(u)\in(m_0, m_0+\hat{m})$. This completes the proof.
\end{proof}

\begin{proof}[{\rm \textbf{Proof of Theorem \ref{cor:fixed-AV} and Theorem \ref{cor:fixed-K}}}]
   By the argument in Remark \ref{bigremark}, the proof follows immediately from Theorem \ref{highenergy}.
\end{proof}

\section{Existence of a ground state solution}
From Corollary \ref{cpt1},  $(PS)_d$ condition is satisfied if the energy level $d$ is strictly lower than the compactness threshold $m_0$. Next, we will prove that the least energy  $m_A$ is indeed strictly lower than $m_0$ and we use a minimization method to obtain the existence of a ground state solution for \eqref{problem}.

Recall that the sequence $\{y_n\}\subset\mathbb{R}^3$ constructed in the proof of Lemma \ref{le:0<mAleqm0} satisfies $|y_n|\to\infty$ as $n\to\infty$ and $t_n\omega_{y_n}\in\mathcal{N}_A$ for every $n\in \N$.
\begin{lemma}\label{strictlysmaller}
    Suppose that $(H_2)$ holds. Then there exists $c_2>0$ and $M>0$ such that
    \begin{equation}\label{ma<m0}
        E_A(t_n\omega_{y_n})\leq m_0-c_2e^{-\sigma|y_n|},\quad {\rm for~}|y_n|\geq M.
    \end{equation}
    In particular, $m_A<m_0$.
\end{lemma}
\begin{proof}
  Considering the sequence $u_n = t_n \omega_{y_n}$ constructed in the proof of Lemma \ref{le:0<mAleqm0}, we have
\begin{equation}\label{inizi}
  \begin{aligned}
      E_A(u_n)&=\frac{t_n^2}{2}\Vert\omega_{y_n}\Vert^2+\frac{t_n^4}{4}\int_{\mathbb{R}^3}K(x)\phi_{|\omega_{y_n}|}|\omega_{y_n}|^2dx-\frac{t_n^p}{p}\int_{\mathbb{R}^3}|\omega_{y_n}|^p\,dx\\
      &=E_0(t_n\omega_{y_n})+\frac{t_n^2}{2}\int_{\mathbb{R}^3}\left(|A(x+y_n)|^2+V(x+y_n)-V_\infty\right)|\omega|^2\,dx+\frac{t_n^4}{4}\int_{\mathbb{R}^3}K(x)\phi_{|\omega_{y_n}|}|\omega_{y_n}|^2dx\\
      &\leq m_0+\frac{t_n^2}{2}\left(\int_{\mathbb{R}^3}\left(|A(x+y_n)|^2+V(x+y_n)-V_\infty\right)|\omega|^2\,dx+C\int_{\mathbb{R}^3}K(x)\phi_{|\omega_{y_n}|}|\omega_{y_n}|^2dx\right).
  \end{aligned}
  \end{equation}
 By $(H_2)$, one has
\begin{equation}\label{44}
\begin{aligned}
     &\quad\int_{\mathbb{R}^3}-\left(|A(x+y_n)|^2+V(x+y_n)-V_\infty\right)|\omega|^2\,dx\\
     &\geq  \int_{B_1(0)}-\left(|A(x+y_n)|^2+V(x+y_n)-V_\infty\right)|\omega|^2\,dx\\
     &\geq\inf_{x\in B_1(0)}\left(-\left(|A(x+y_n)|^2+V(x+y_n)-V_\infty\right)\right)\int_{B_1(0)}|\omega|^2\,dx\\
     &\geq Ce^{-\sigma|y_n|}
\end{aligned}
\end{equation}
for some $C>0$ independent of $n$.

On the other hand, by the H\"older inequality and the Hardy inequality, we obtain
\begin{equation}\label{Hardyinequality}
  \begin{aligned}
    \left|\phi_{|\omega_{y_n}|}\right|_\infty&=\frac{1}{4\pi}\sup_{x\in\mathbb{R}^3}\left(\int_{\mathbb{R}^3}\frac{K(z)||\omega_{y_n}(z)|^2}{|x-z|}\,dz\right)\\
    &=\frac{1}{4\pi}\sup_{x\in\mathbb{R}^3}\left(\int_{\mathbb{R}^3}\frac{K(x-z)||\omega_{y_n}(x-z)|^2}{|z|}\,dz\right)\\
    &\leq \frac{|\omega|_\infty}{4\pi}\sup_{x\in\mathbb{R}^3}\left(\int_{\mathbb{R}^3}\frac{K(x-z)||\omega_{y_n}(x-z)|}{|z|}\,dz\right)\\
    &\leq \frac{|\omega|_\infty}{4\pi}\sup_{x\in\mathbb{R}^3}\left(\left(\int_{\mathbb{R}^3}\frac{|\omega_{y_n}(x-z)|^2}{|z|^2}\,dz\right)^\frac{1}{2}\left(\int_{\mathbb{R}^3}|K(x-z)|^2\,dz\right)^\frac{1}{2}\right)\\
    &\leq C|\nabla\omega_{y_n}|_2\left|K\right|_2= C(|\nabla \omega|_2,|K|_2).
\end{aligned}  
\end{equation}
By the decay of $\omega$, we know that $|\omega|^2\leq Ce^{-2\sqrt{V_\infty}|x|}$. Moreover, $K\in L^1(\mathbb{R}^3, \mathbb{R})$ by Remark \ref{integra}, so we want to apply Lemma \ref{IntergalLemma} with $g(x)=e^{-2\sqrt{V_\infty}|x|}$, $h=K$, $\kappa=2\sqrt{V_\infty}$, $\nu=0$ and $\gamma=1$, so that
\begin{align*}
 \quad \int_{\mathbb{R}^3}K(x)e^{2\sqrt{V_\infty}|x|}dx<\infty.
\end{align*}
For this purpose we now choose\footnote{Note that such a choice was not necessary so far.} $y_n$ in a special way: we fix $z\in \R$ with $|z|=1$ and $y_n=-nz$. In this way Lemma \ref{IntergalLemma} reads as
\[
\lim_{n\to \infty}\left(\int_{\R}e^{-2\sqrt{V_\infty}|x|}K(x-nz)dx\right)e^{2\sqrt{V_\infty}|x|}\int_{\R}K(x)e^{-2\sqrt{V_\infty}x\cdot z}dx,
\]
so that
\begin{equation}\label{approx}
\int_{\R}e^{-2\sqrt{V_\infty}|x|}K(x-nz)dx\leq Ce^{-2n\sqrt{V_\infty}}
\end{equation}
for some $C$.
Then, by \eqref{Hardyinequality}, we have
\begin{equation}\label{55}
    \begin{aligned}
        \int_{\mathbb{R}^3}K(x)\phi_{|\omega_{y_n}|}(x)|\omega_{y_n}|^2dx&\leq\left|\phi_{|\omega_{y_n}|}\right|_\infty\int_{\mathbb{R}^3}K(x)|\omega_{y_n}|^2dx\\
        &\leq C\int_{\mathbb{R}^3}K(x+y_n)|\omega|^2dx\\
        &\leq Ce^{-2\sqrt{V_\infty}|y_n|}.
    \end{aligned}
\end{equation}
 Note that $\{t_n\}$ is bounded. Combining \eqref{inizi}, \eqref{44}, \eqref{55} and using the fact that $0<\sigma<2\sqrt{V_\infty}$, we obtain \eqref{ma<m0}, as claimed.
\end{proof}
	
\begin{proof}[{\rm \textbf{Proof of Theorem \ref{leastenergy}}}] 
This proof follows directly from Lemma \ref{le:0<mAleqm0}, Lemma \ref{b>m0} and Lemma \ref{strictlysmaller}.
\end{proof}	

\section{Intertwining solutions }\label{secint}

In this section, we investigate the existence of intertwining solutions under symmetry assumptions $(A_2), (V_2)$ and $(K_2)$.

Let us start introducing the precise setting. $G$ denotes a closed subgroup of the group $O(3)$ of orthogonal transformations on $\R$ and let $\eta:G\rightarrow \mathbb{S}^1$ be a group homomorphism into the unit complex numbers.	The action of $G$ on $\HA$ is given by $(g,u)=u_g$, where
	\begin{equation*}
		u_g(x):=\eta(g)u(g^{-1}x)
	\end{equation*}
is such that
	\begin{equation*}
		\langle u_g,v_g\rangle _A=\langle u,v\rangle _A
	\end{equation*}
	for all $u,v\in\HA$ and $g\in G$. Thus, it is easy to check that the functional $E_A$ is $G$-invariant. Define the fixed point space of the $G$-action as
	\begin{align*}
		\HA^\eta:&=\{u\in \HA:u_g(x)=u(x)\ \forall g\in G, \ \forall x\in\R\}\\
		&=\{u\in \HA:u(gx)=\eta(g)u(x)\ \forall g\in G, \ \forall x\in\R\}.
	\end{align*}
	By the principle of symmetric criticality by Palais \cite{Palais1979}, the critical points of the functional $E_A$ restricted to $\HA^\eta$ are intertwining solutions of \eqref{problem}. Moreover, all non-trivial critical points of $_A$ constrained on $\HA^\eta$ lie on the Nehari manifold
	\begin{equation*}\label{NAeta}
		\mathcal{N}_A^\eta:=\{u\in\HA^\eta\setminus\{0\}:E_A^\prime(u)u=0\}.
	\end{equation*}
It is standard to check that the functional $E_A$ is bounded from below on $\mathcal{N}_A^\eta$, so we can define the least energy
	\begin{equation*}\label{mAeta}
		m_A^\eta:=\inf_{u\in\mathcal{N}_A^\eta}E_A(u).
	\end{equation*}
Again by a classical argument, $m_A^\eta>0$ and $\mathcal{N}_A^\eta$ possesses some properties  similar to Lemma \ref{properties}.
	
Inspired by \cite[Lemma 2.4]{Hirata2008},	we obtain the precise compactness result for symmetric $(PS)$ sequence in our setting.
\begin{lemma}\label{compactness}
		Let $\{u_n\}\subset\HA^\eta$ be  a $(PS)_d$ sequence of the functional $E_A$ with $d>0$. Then the following results hold:
		\begin{enumerate}[{\rm (a)}]
			\item If $G$ is an infinite group satisfying $(G_1)$, then $\{u_n\}$ has a strongly convergent subsequence.
			\item  If $G$ is a finite group satisfying $(G_2)$ and $d<\ell(G)m_0$, then $\{u_n\}$ has a strongly convergent subsequence.
		\end{enumerate}
	\end{lemma}
	\begin{proof}
		First, we prove (a).
		From Lemma \ref{PSbounded}, we know that $\{u_n\}$ is bounded. According to Lemma \ref{splitting}, there exists $u^0\in\HA$ such that $u_n\rightharpoonup u^0$ in $\HA$ and $u^0$ is a solution of \eqref{problem}. Suppose by contradiction that $u_n\not\rightarrow u^0$ in $\HA$. It follows from Lemma \ref{splitting} that there exists a number $k\in\mathbb{N}^+$, $k$ sequences of points $(y_n^j)\subset\R$ with $|y_n^j|\rightarrow\infty$ and $|y_n^j-y_n^i|\rightarrow\infty$  for $j\neq i$, $1\leq j\leq k$ and $k$ nontrivial solutions $u^1,u^2,\cdots,u^k\in\H$ of \eqref{limitproblem}
		such that
		
        \begin{equation}\label{strongconvergent}
			u_n-u^0-\sum_{i=1}^{k}u^i(\cdot-y_n^i)\rightarrow 0~ {\rm in}~\H
		\end{equation}
		and
		\begin{equation}\label{d}
			d=	E_A(u^0)+\sum_{i=1}^{k}E_0(u^i).
		\end{equation}
Now, set 
		\begin{equation*}
			r_n^j:=\frac{y_n^j}{|y_n^j|}\in S^{2},\quad\mathrm{for~}j=1,2,\cdots,k.
		\end{equation*}
Thus, there exists $r_\infty^j\in S^{2}$ such that, up to subsequences, $r_n^j\rightarrow r_\infty^j$ as $n\rightarrow\infty$.

We claim that the set $\{r_\infty^j:j=1,2,\cdots,k\}$ is $G$-symmetric. Indeed, since both $u_n$ and $u^0$ belong to $\HA^\eta\subset\HA$,  for any  $j\in\{{1,2,\cdots,k}\}$ and $g\in G$, we know that
		\begin{equation*}
			u_n(gx)=\eta(g)u_n(x) \quad\mathrm{and}\quad u^0(gx)=\eta(g)u^0(x)\quad\mathrm{for}~x\in B_1(y_n^j).
		\end{equation*} 
		Due to the fact that $|y_n^j-y_n^i|\rightarrow\infty$ if $j\neq i$ and since for every $a,b\in\mathbb{R}$ we have $\frac{1}{2}|a|^2\leq |a+b|^2+|b|^2$, there exists a constant $\varrho>0$ such that for every $n\in\N$
		\begin{align*}
			\int_{B_1(y_n^j)}\left|\sum_{i=1}^{k}u^j(\cdot-y_n^i)\right|^2dx
			&\geq \int_{B_1(y_n^j)}\left(\frac{1}{2}|u^j(\cdot-y_n^j)|^2-\left|\sum_{i\neq j}^{k}u^j(\cdot-y_n^i)\right|^2\right)dx\\
			&= \int_{B_1(0)}\left(\frac{1}{2}|u^j|^2-\left|\sum_{i\neq j}^{k}u^j(\cdot+y_n^j-y_n^i)\right|^2\right)dx\\
			&\geq\varrho+o_n(1).
		\end{align*}
		Thus, from \eqref{strongconvergent} one has
		\begin{align*}
			\liminf_{n\rightarrow\infty}\int_{B_1(y_n^j)}|u_n-u^0|^2dx&=\liminf_{n\rightarrow\infty}\int_{B_1(y_n^j)}\left|\sum_{i=1}^{k}u^i(\cdot-y_n^i)\right|^2dx\geq\varrho>0.
		\end{align*}
		Moreover, we also have
\begin{equation}\label{soprasei}
\begin{aligned}
			\liminf_{n\rightarrow\infty}\int_{B_1(y_n^j)}|u_n(gx)-u^0(gx)|^2dx&=\liminf_{n\rightarrow\infty}\int_{B_1(y_n^j)}|\eta(g)u_n(x)-\eta(g)u^0(x)|^2dx\\
			&=	\liminf_{n\rightarrow\infty}\int_{B_1(y_n^j)}|u_n-u^0|^2dx
			\geq\varrho>0.
		\end{aligned}
\end{equation}

Next, we will show that for any $j\in\{1,2,\cdots,k\}$, there exists $q\in\{1,2,\cdots,k\}$ such that the sequence $\{gy_n^j-y_n^q\}$ is bounded. We argue by contradiction, so suppose that for any $q\in\{1,2,\cdots,k\}$ the sequence $\{gy_n^j-y_n^q\}$ is unbounded. Without loss of generality, we assume that $|gy_n^j-y_n^q|\rightarrow\infty$ as $n\rightarrow\infty$ for any $q\in\{1,2,\cdots,k\}$. Recalling that $u_n$ and $u^0$ belong to $\HA^\eta$, by \eqref{soprasei} and \eqref{strongconvergent}, we have 
		\begin{align*}
			0<\varrho\leq\liminf_{n\rightarrow\infty}\int_{B_1(y_n^j)}|u_n(gx)-u^0(gx)|^2dx&=\liminf_{n\rightarrow\infty}\int_{B(y_n^j)}\left|\sum_{i=1}^{k}u^i(gx-y_n^i)\right|^2dx
			\\
			&=\liminf_{n\rightarrow\infty}\int_{B_1(0)}\left|\sum_{i=1}^{k}u^i(gx+gy_n^j-y_n^i)\right|^2dx.
		\end{align*}
But $|u^j|\in L^2(\R,\mathbb{R})$ for $1\leq i\leq k$, so we obtain that
		\begin{equation*}
			\liminf_{n\rightarrow\infty}\int_{B_1(0)}\left|\sum_{i=1}^{k}u^i(gx+gy_n^j-y_n^i)\right|^2dx=0,
		\end{equation*}
which is a contradiction. Hence,  there exists $q\in\{1,2,\cdots,k\}$ such that the sequence $\{gy_n^j-y_n^q\}$ is bounded. It follows from $|y_n^j|\rightarrow\infty$ and $|y_n^q|\rightarrow\infty$ that
		\begin{equation*}
			\lim_{n\rightarrow\infty}\dfrac{|gy_n^j-y_n^q|}{|y_n^j|}=0\quad\mathrm{and}\quad \lim_{n\rightarrow\infty}\dfrac{|gy_n^j-y_n^q|}{|y_n^q|}=0.
		\end{equation*}
		Then one has
		\begin{equation*}
			\lim_{n\rightarrow\infty}\dfrac{|y_n^j|}{|y_n^q|}=\lim_{n\rightarrow\infty}\left|g\dfrac{y_n^j}{|y_n^q|}\right|=\lim_{n\rightarrow\infty}\left|\dfrac{y_n^q}{|y_n^q|}\right|=1,
		\end{equation*}
but also
\[
\lim_{n\to\infty}\frac{gy_n^j}{|y_n^j|}=\lim_{n\to\infty}\frac{y_n^q}{|y_n^j|} \mbox{ and }\lim_{n\to\infty} \dfrac{gy_n^j}{|y_n^q|}=\lim_{n\to\infty}\dfrac{y_n^q}{|y_n^q|}.
\]
Therefore, we obtain that
		\begin{equation*}
			gr_\infty^j=\lim_{n\to\infty}g\frac{y_n^j}{|y_n^j|}=\lim_{n\to\infty}\frac{gy_n^j}{|y_n^j|}=\lim_{n\to\infty}\frac{y_n^q}{|y_n^j|}=\lim_{n\to\infty}\frac{|y_n^q|}{|y_n^j|}\frac{y_n^q}{|y_n^q|}=r_\infty^q,
		\end{equation*}
which proves that the set $\{r_\infty^j:j=1,2,\cdots,k\}$ is $G$-symmetric. Then we know that
		\begin{equation*}
			\ell(G)\leq\min\{\#Gr_\infty^j:1\leq j\leq k\}\leq\#\{r_\infty^q:1\leq q\leq k\}\leq k<\infty.
		\end{equation*}
		This is a contradiction with $\ell(G)=\infty$, which concludes the proof of (a). 
		
Next, we show (b). Starting as before, since  $u^0$ is a solution of \eqref{problem}, then either $u^0=0$ or $u$ belongs to the Nehari manifold. In any case, $E_A(u^0)\geq0$.
Since $E_0(u^i)\geq m_0$, from \eqref{d} we have
		\begin{equation*}
			d=	E_A(u^0)+\sum_{i=1}^{k}E_0(u^i)\geq km_0\geq \ell(G)m_0,
		\end{equation*}
		which is contradiction with $d<\ell(G)m_0$ and the proof of (b) is completed.
	\end{proof}

    	\begin{proof}[{\rm \textbf{ Proof of Theorem \ref{existence1}}}]
		Using Lemma \ref{compactness} (a) and an argument  similar to that in the proof of Lemma \ref{b>m0}, the functional $E_A$ has a nontrivial critical point $u\in\HA^\phi$ with $E_A(u)=m_A^\eta$. 
	\end{proof}

At this point, we recall that the $\mathbb{S}^1$-action on $\HA^{\eta}$ is defined by scalar multiplication of unit complex number $u\mapsto e^{i\theta}u$  for any $\theta\in\mathbb{R}/2\pi\mathbb{Z}$ and the $\mathbb{S}^1$-orbit of $u$  is $\left\{e^{i\theta}u:\theta\in\mathbb{R}/2\pi\mathbb{Z}\right\}$. Of course, the functional $E_A$ is invariant under the $\mathbb{S}^1$-action.

In the following, we give a modified version of $\mathbb{S}^1$-symmetric mountain pass theorem  from \cite[Theorem 1.5]{CP1991}.
\begin{proposition}
\label{prop:smp} 
Let $X$ be a infinite dimension Banach space with a continuous group action of $\mathbb{S}^1$ and $I$ be a $\mathbb{S}^1$-invariant functional which satisfies $(PS)_d$ condition, $d>0$. Suppose that the functional $I$ satisfies the following mountain pass structure:
\begin{enumerate}[{\rm (MP$_1$)}]
    \item $I(0)=0$ and there exists a linear subspace $\hat{X}$ of $X$ of finite codimension and  $r>0$ such that $I(u)>0$ for all $u\in \hat{X}$ with $\Vert u\Vert=r$;
    \item For each finite dimensional subspace $\hat{E}$ of $X$, there is $r(E)>0$ such that $I(u)\leq 0$ for all $u\in E$ with $\Vert u\Vert\geq r(   E)$;
    \item $I(u)\leq 0$ for all $u\in {\rm Fix}(\mathbb{S}^1):=\{u\in X:gu=u~\mathrm{for~all}~g\in\mathbb{S}^1\}$.
\end{enumerate}
    Then functional $I$ possesses a sequence of critical values which is unbounded above.
\end{proposition}

		\begin{proof}[{\rm \textbf{ Proof of Theorem \ref{multiplicity}}}]
	 From \cite{CCS2012}, we know that $\HA^\eta$ is an infinite dimension space.	  By Lemma \ref{compactness}, the functional $E_A$ satisfies the $(PS)_d$ condition for $d>0$. Recalling the action of 
 $\mathbb{S}^1$ on 	$\HA^\eta$ and the fact  that the functional  $E_A$ is an $\mathbb{S}^1$-invariant functional, we can consider the fixed point set under the $\mathbb{S}^1$ action
    \begin{equation*}
{\rm Fix}(\mathbb{S}^1)=\{u\in\HA^\eta:gu=u,~\forall g\in\mathbb{S}^1\}=\{0\}.
    \end{equation*}
 In particular,  (MP$_3$) holds. Next, we  just need to show that  the functional $E_A$ satisfies the mountain pass conditions (MP$_1$) and (MP$_2$).

First, from the continuous embedding
	$\HA\hookrightarrow L^t(\mathbb{R}^3,\mathbb{C})$  for $t\in[2,6]$, we have
\begin{equation*}
    E_A(u)
		=\dfrac{1}{2}\|u\|_A^2+\frac{1}{4}\int_{\mathbb{R}^3}K(x)\phi_{|u|}|u|^2\,dx-\dfrac{1}{p}\int_{\mathbb{R}^3}|u|^pdx\geq \dfrac{1}{2}\|u\|_A^2-\dfrac{C_p}{p}\|u\|_A^p.
\end{equation*}
Therefore,  there exist $\alpha>0$ and sufficiently small $\rho>0$ such that $E_A(u)\geq\alpha>0$ for $u\in\HA^\eta$ and $\Vert u\Vert_A=\rho$, and so (MP$_1$) holds.

 Then, let $\hat{E}$ be  a finite dimensional subspace of $\HA^\eta$.    Since all norms in $\hat{E}$ are equivalent, then there exists $\gamma_p>0$ such that
    \begin{equation*}
        \gamma_p\Vert u\Vert_A\leq \Vert u\Vert_p,\quad \mathrm{for}~u\in E.
    \end{equation*}
    For $u\in \hat{E}\setminus\{0\}$, using \eqref{S-1}, we have
\begin{equation*}
    \begin{aligned}
            E_A(u)
		&=\dfrac{1}{2}\|u\|_A^2+\frac{1}{4}\int_{\mathbb{R}^3}K(x)\phi_{|u|}|u|^2\,dx
        -\dfrac{1}{p}\int_{\mathbb{R}^3}|u|^pdx\\
        &\leq  \dfrac{1}{2}\|u\|_A^2+\dfrac{1}{4}S^{-3}\left|K\right|_2^2\Vert u\Vert_A^4-\dfrac{(\gamma_p)^p}{p} \|u\|_A^p, 
    \end{aligned}
    \end{equation*}
    which yields that $E_A(u)\rightarrow-\infty$ as $\Vert u\Vert_A\rightarrow\infty$. Hence, (MP$_2$) holds true.

In conclusion, we apply Proposition \ref{prop:smp} to obtain a critical point sequence $\{u_n\}$ of \eqref{problem} satisfying $E_A(u_n)\rightarrow\infty$ as $n\rightarrow\infty$ and these critical points are geometrically distinct (see Definition \ref{dist}) because of the difference of their functional values.
	\end{proof}

Finally, we consider a finite group $G$ satisfying  $(G_2)$. By constructing a proper test--function and estimating its energy, we obtain the range of the least energy of \eqref{problem} and give the proof of Theorem \ref{existence2}.
    
Fix a positive constant
\begin{equation}\label{alfa}
 \epsilon\in\left(0,\frac{d_G\sqrt{V_\infty}-\alpha}{d_G\sqrt{V_\infty}+\alpha}\right),
 \end{equation}
 where $\alpha$ comes from assumption $(H_3)$ and $d_G$ is defined as \eqref{dG}.
Take a smooth cut-off function $f\in C^\infty([0,\infty),[0,1])$ such that $f(t)\equiv1$ for $t\leq 1-\epsilon$ and $f(t)\equiv0$ for $t\geq 1$\footnote{We could also require that $|f'|\leq 2/\epsilon$, though we will not need this condition.}. Recalling that $\omega$ is a positive least energy solution of \eqref{limitproblemR} which is radially symmetric and decreasing in the radial direction, for $R>0$ we set
	\begin{equation*}
		\omega^R(x):=f\left(\dfrac{|x|}{R}\right)\omega(x).
	\end{equation*}
	
	\begin{lemma}\label{deacyestimate1}
		As $R\rightarrow\infty$, the following estimates hold:
		\begin{equation*}
			\I \left||\nabla\omega^R|^2-|\nabla\omega|^2\right|dx=O(e^{-2\sqrt{V_\infty}R(1-\epsilon)}),
		\end{equation*}
		\begin{equation*}
			\I \left||\omega^R|^p-|\omega|^p\right|dx=O(e^{-p\sqrt{V_\infty}R(1-\epsilon)}).
		\end{equation*}
	\end{lemma}
	\begin{proof}
		It follows from	 \eqref{decayestimate} that as $R\rightarrow\infty$,  we have
		\begin{equation*}
			\begin{aligned}
				\I \left||\nabla\omega^R|^2-|\nabla\omega|^2\right|dx
				&=\int_{|x|\geq R(1-\epsilon)} \left||\nabla\omega^R|^2-|\nabla\omega|^2\right|dx
				\\
				&\leq \int_{|x|\geq R(1-\epsilon)} \left|\dfrac{1}{R}\omega f'\left(\dfrac{|x|}{R}\right)\frac{x}{|x|}+f\left(\dfrac{|x|}{R}\right)\nabla\omega\right|^2dx+\int_{|x|\geq R(1-\epsilon)}|\nabla\omega|^2dx\\
				&\leq C\int_{|x|\geq R(1-\epsilon)}(1+|x|)^{-2}e^{-2\sqrt{V_\infty}|x|}dx\\
				&=C\int_{R(1-\epsilon)}^{\infty}r^2(1+r)^{-2}e^{-2\sqrt{V_\infty}r}dr\\
				&\leq 	C\int_{R(1-\epsilon)}^{\infty}e^{-2\sqrt{V_\infty}r}dr=O(e^{-2\sqrt{V_\infty}R(1-\epsilon)})
			\end{aligned}
		\end{equation*}
		and similarly
		\begin{equation*}
			\begin{aligned}
				\I \left||\omega^R|^p-|\omega|^p\right|dx&=\int_{|x|\geq R(1-\epsilon)}\left||\omega^R|^p-|\omega|^p\right|dx\\
				&\leq \int_{|x|\geq R(1-\epsilon)}\left|f\left(\dfrac{|x|}{R}\right)\omega\right|^pdx+\int_{|x|\geq R(1-\epsilon)}\left|\omega\right|^pdx\\
				&\leq C\int_{|x|\geq R(1-\epsilon)}(1+|x|)^{-p}e^{-p\sqrt{V_\infty}|x|}dx\\
				&=C\int_{R(1-\epsilon)}^{\infty}r^2(1+r)^{-p}e^{-p\sqrt{V_\infty}r}dr\\
				&\leq 	C\int_{R(1-\epsilon)}^{\infty}e^{-p\sqrt{V_\infty}r}dr=O(e^{-p\sqrt{V_\infty}R(1-\epsilon)}),
			\end{aligned}
		\end{equation*}
since $p>4$. 
	\end{proof}
	
Now, define
	\begin{equation}\label{lambdaR}
		\lambda_R:=\left(\dfrac{\alpha+d_G\sqrt{V_\infty}}{2d_G\sqrt{V_\infty}}\right)\dfrac{d_G}{2}R<\dfrac{d_G}{2}R.
	\end{equation} 
	
	\begin{lemma}\label{symmetricmA<m0}
    Suppose that $(H_3)$ holds. Then
		there exist $c_3>0$ and $\xi>0$ such that for $j=1,2,\cdots,\ell(G)$,  
		\begin{equation*}
			E_A(t\omega^{\lambda_R}(\cdot-Re_j))\leq m_0-c_3e^{-\alpha R}, \quad\mathrm{for}~t\geq 0~\mathrm{and}~R\geq\xi,
\end{equation*}
where $e_j$ is defined in \eqref{e}.
	\end{lemma}
	\begin{proof}
    Since $\omega$ is a real-valued function, then $\omega\in\HA$.
By Lemma \ref{properties} (d), there exists $t_R>0$ such that $t_R\omega^{\lambda_R}(\cdot-Re_j)\in\mathcal{N}_A$, so that we can easily find
\[
    t_R^{p-2}=\frac{\displaystyle \|\omega^{\lambda_R}(\cdot-Re_j)\|_A^2+t_R^2\int_{\R}K(x)\phi_{|\omega^{\lambda_R}(\cdot-Re_j)|}|\omega^{\lambda_R}(x-Re_j)|^2dx}{|\omega^{\lambda_R}(\cdot-Re_j)|_p^p}
\]
As $R\to\infty$ (and so $\lambda_R\to \infty$, as well), we have, after a change of variables in the related integrals,
\begin{equation*}
\|\omega^{\lambda_R}(\cdot-Re_j)\|_A^2\to \|\omega\|^2_\bullet,\quad |\omega^{\lambda_R}(\cdot-Re_j)|_p^p\to |\omega|_p^p,
\end{equation*}
and 
\begin{equation*}
    \int_{\R}K(x)\phi_{|\omega^{\lambda_R}(\cdot-Re_j)|}|\omega^{\lambda_R}(x-Re_j)|^2dx\to0.
\end{equation*}

From this, it is easy to verify that there exist $t_2>t_1>0$ independent of $R$ such that $t_R\in[t_1,t_2]$. Thus,
		\begin{equation}\label{eq:maxmax}
			\max_{t\geq0}E_A(t\omega^{\lambda_R}(\cdot-Re_j))=E_A(t_R\omega^{\lambda_R}(\cdot-Re_j))=\max_{t\in[t_1,t_2]}E_A(t\omega^{\lambda_R}(\cdot-Re_j))
		\end{equation}
		for $R$ large enough.
For $t\in[t_1,t_2]$ and sufficiently large $R$, so that $\lambda_R>1$, by $(H_3)$ one has
\begin{equation}\label{nomina}
			\begin{aligned}
				&\quad\I \left(|A(x)|^2+V(x)-V_\infty\right)(t\omega^{\lambda_R}(x-Re_j))^2dx\\
				&=\int_{|x|\leq\lambda_R} \left(|A(x+Re_j)|^2+V(x+Re_j)-V_\infty\right)(t\omega^{\lambda_R})^2dx\\
				&\leq\int_{|x|\leq\lambda_R} \left(|A(x+Re_j)|^2+V(x+Re_j)-V_\infty\right)(t\omega)^2dx\\
				&\leq-c_1t_1^2\int_{|x|\leq\lambda_R}e^{-\alpha|x+Re_j|}\omega^2 dx\\
				&\leq-c_1t_1^2\left(\int_{|x|\leq 1}e^{-\alpha|x|}\omega^2 dx\right)e^{-\alpha R}\\
				&\leq -O(e^{-\alpha R}).
			\end{aligned}
\end{equation}
        Besides, $(H_3)$ implies that $K\in L^1(\mathbb{R}^3, \mathbb{R})$. Now we apply Lemma \ref{IntergalLemma} with $g(x)=e^{-2\sqrt{V_\infty}|x|}$, $h=K$, $\kappa=2\sqrt{V_\infty}$ and $\nu=0$, so that
\begin{align*}
    \gamma=1,\quad \mbox{ and }\quad \int_{\mathbb{R}^3}K(x)e^{2\sqrt{V_\infty}|x|}<\infty.
\end{align*}
Therefore, for $R$ large enough, by \eqref{decayestimate} and reasoning as for \eqref{approx}, we have
\begin{equation}\label{cambi}
    \int_{\mathbb{R}^3}K(x+Re_j)|\omega|^2\,dx\leq C\int_{\mathbb{R}^3}K(x+Re_j)e^{-2\sqrt{V_\infty}|x|}\leq Ce^{-2\sqrt{V_\infty}R}=O(e^{-2\sqrt{V_\infty}R}).
\end{equation}
Similarly to \eqref{Hardyinequality}, we also know that $$|\phi_{|\omega(\cdot-Re_j)|}|_\infty\leq C.$$ Then from \eqref{cambi} one has
\begin{equation}\label{sop}
    \begin{aligned}
        \int_{\mathbb{R}^3}K(x)\phi_{|\omega^{\lambda_R}(x-Re_j)|}{|\omega^{\lambda_R}(x-Re_j)|}^2dx&\leq \int_{\mathbb{R}^3}K(x)\phi_{|\omega(\cdot-Re_j)|}{|\omega(x-Re_j)|}^2dx\\
        &\leq\left|\phi_{|\omega(\cdot-Re_j)|}\right|_\infty\int_{\mathbb{R}^3}K(x)|\omega(x-Re_j)|^2dx\\
        &\leq C\int_{\mathbb{R}^3}K(x+Re_j)|\omega|^2dx=O(e^{-2\sqrt{V_\infty}R}).
    \end{aligned}
\end{equation}
		Thus, it follows from \eqref{nomina}, \eqref{sop} and Lemma  \ref{deacyestimate1} that for $t\in[t_1,t_2]$ and sufficiently large $R$,
\begin{equation}\label{fini}
\begin{aligned}
E_A(t\omega^{\lambda_R}(\cdot-Re_j))&=\dfrac{1}{2}\Vert t\omega^{\lambda_R}(\cdot-Re_j)\Vert_A^2+\dfrac{1}{4}\int_{\mathbb{R}^3}K(x)\phi_{|t\omega^{\lambda_R}(\cdot-Re_j)|}{|t\omega^{\lambda_R}(\cdot-Re_j)|}^2dx\\
&\quad-\dfrac{1}{p}\I |t\omega^{\lambda_R}(\cdot-Re_j)|^pdx\\
&\leq\dfrac{1}{2}\I \left(|\nabla(t\omega)|^2+V_\infty (t\omega)^2\right)dx-\dfrac{1}{p}\I |t\omega|^pdx\\
&\quad +O(e^{-2\sqrt{V_\infty}\lambda_R(1-\epsilon)})-O(e^{-\alpha R})+O(e^{-2\sqrt{V_\infty}R})\\
&\leq\max_{t\geq0}E_0(t\omega)+O(e^{-2\sqrt{V_\infty}\lambda_R(1-\epsilon)})-O(e^{-\alpha R})+O(e^{-2\sqrt{V_\infty}R})\\
&=m_0+O(e^{-2\sqrt{V_\infty}\lambda_R(1-\epsilon)})-O(e^{-\alpha R})+O(e^{-2\sqrt{V_\infty}R}).
\end{aligned}
\end{equation}
By $(H_3)$ and \eqref{dG} we find $\alpha R<d_G\sqrt{V_\infty}R\leq 2\sqrt{V_\infty}R$ and from \eqref{lambdaR} we also get
\[
R=\frac{2d_G\sqrt{V_\infty}}{\alpha+d_G\sqrt{V_\infty}}\lambda_R<2\lambda_R.
\]
Moreover, from \eqref{alfa} we know that
\[
\alpha<d_G\sqrt{V_\infty}\frac{1-\epsilon}{1+\epsilon},
\]
so that
\[
\alpha R<2d_G\sqrt{V_\infty}\lambda_R(1-\epsilon)\frac{R}{2\lambda_R(1+\epsilon)}<2d_G\sqrt{V_\infty}\lambda_R(1-\epsilon).
\]
As a consequence, from \eqref{fini} we can find  $c_3>0$ and $\xi>0$ such that for $R\geq\xi$ and $j=1,2,\cdots,\ell(G)$, 
		\begin{equation*}
			E_A(t\omega^{\lambda_R}(\cdot-Re_j))\leq m_0-c_3e^{-\alpha R},\quad \forall t\in[t_1,t_2].
		\end{equation*}
        Thanks to \eqref{eq:maxmax}, we obtain that 
        	\begin{equation*}
			E_A(t\omega^{\lambda_R}(\cdot-Re_j))\leq m_0-c_3e^{-\alpha R},\quad \forall t\geq 0,
		\end{equation*}
		which completes the proof.
	\end{proof}
	
\begin{proof}[{\rm \textbf{ Proof of Theorem \ref{existence2}}}]
Let $R=\xi$, where $\xi$	is given from Lemma \ref{symmetricmA<m0}. Recall the definition of the minimal cardinality  of $G$ and choose $x_0\in S^{2}$ such that 
\eqref{x0} holds. Moreover, choose $g_j\in G$ such that $g_jx_0=e_j$ for every $j=1\ldots, \ell(G)$, see \eqref{e}. Then, we define the  function
	\[
\Upsilon_R(x):=\sum_{j=1}^{\ell(G)}\eta(g_j)\omega^{\lambda_R}(x-Re_j).
\]	
Since  $\omega^{\lambda_R}$ is radially symmetric,  it is easy to check that $\Upsilon_R( gx)=\eta( g)\Upsilon_R(x)$ for all $ g\in G$ and $x\in\R$. 
Since $\lambda_R<d_GR/2$, we immediately see that the functions $\omega(\cdot-Re_i)$ and $\omega(\cdot-Re_j)$ have disjoint support sets if $i\neq j$. In view of Lemma \ref{properties} (d), 
        there exists $t_{\Upsilon_R}>0$ such that $t_{\Gamma_R}\Upsilon_R\in\mathcal{N}_A^\eta$.
        Therefore, from Lemma \ref{symmetricmA<m0}  we have
		\begin{equation*}
			E_A(t_{\Upsilon_R}\Upsilon_R)=\sum_{j=1}^{\ell(G)}E_A(t_{\Upsilon_R}\omega^{\lambda_R}(\cdot-Re_j))< \ell(G)m_0.
		\end{equation*}
Obviously, this implies that $m_A^\eta<\ell(G)m_0$.

Finally, by  definition of $m_A^\eta$, there exists a minimizing sequence $\{u_n\}\subset\mathcal{N}_A^\eta$ such that $E_A(u_n)\rightarrow m_A^\eta $. 
 By Lemma \ref{compactness} (b) and an argument similar to that used in the proof of Lemma \ref{b>m0}, we find that $E_A$ has a nontrivial critical point $u\in\HA^\eta$ with $E_A(u)=m_A^\eta$. 
	\end{proof}

    \noindent{\bf  Acknowledgements.}
This work is supported by the National Natural Science Foundation of China (Grant No. 12571120, 12271508), and Scientific Research Project of Education Department of Jilin Province (Grant
No. JJKH20261620KJ).

D. M. is a member of the {\em Gruppo Nazionale per l'Analisi Ma\-te\-ma\-ti\-ca, la Probabilit\`a e le loro Applicazioni}
(GNAMPA) of the {\em Istituto Nazionale di Alta Matematica} (INdAM) and he is  partly supported by FFABR {\it Fondo per il finanzia\-mento delle attivit\`a base di ricerca} 2017.
	
    \hspace*{\fill} \\
\noindent{\bf Data availability.}	The manuscript has no associated data.

	\hspace*{\fill} \\
\noindent{\bf Conflict of interest.}	The authors declare that they have no conflict of interest.

	\bibliographystyle{abbrv}
	\bibliography{reference}
\end{document}